\documentclass[11pt]{amsart}

\usepackage{amssymb, array, verbatim, amscd, color}
\usepackage{amsmath, amsfonts,amsthm}
\usepackage{enumerate}
\usepackage[all]{xy}
\usepackage{xy}
\usepackage{hyperref}
\usepackage{wasysym} 

\newtheorem{theorem}{Theorem}[section]
\newtheorem{proposition}[theorem]{Proposition}
\newtheorem{lemma}[theorem]{Lemma}

\newtheorem{corollary}[theorem]{Corollary}

\newtheorem{theorem*}{Theorem}

\theoremstyle{definition}
\newtheorem{definition}[theorem]{Definition}
\newtheorem{example}[theorem]{Example}

\theoremstyle{remark}
\newtheorem{remark}[theorem]{Remark}

\newcommand{\CA}{{\mathcal A}}
\newcommand{\CB}{{\mathcal B}}

\newcommand{\CE}{{\mathcal E}}

\newcommand{\CI}{{\mathcal I}}
\newcommand{\CJ}{{\mathcal J}}

\newcommand{\CM}{{\mathcal M}}

\newcommand{\CT}{{\mathcal{T}}}

\newcommand{\Z}{{\mathbb Z}}

\newcommand{\X}{{\mathcal X}}

\newcommand{\frakF}{{\mathfrak{F}_n}}
\newcommand{\frakE}{{\mathfrak{E}_n}}
\newcommand{\frakD}{\mathfrak{D}}
\newcommand{\frakG}{{\mathfrak{G}_n}}
\newcommand{\frakH}{{\mathfrak{H}_n}}

\newcommand{\wrt}{with respect to }
\newcommand{\PEF}{{\Phi}_\frakE(\frakF)}
\newcommand{\CPEF}{{\Phi}^\frakE(\frakF)}
\newcommand{\CPEG}{{\Phi}^\frakE(\frakG)}

\newcommand{\Ph}{{{\Phi}_n}}
\newcommand{\PB}{{\mathrm{PB}}}
\newcommand{\PO}{{\mathrm{PO}}}
\newcommand{\proj}{\text{-}\mathrm{proj}}
\newcommand{\inj}{\text{-}\mathrm{inj}}

\newcommand{\Eproj}{{\mathfrak E}_n\text{-}\mathrm{proj}}
\newcommand{\Einj}{{\mathfrak E}_n\text{-}\mathrm{inj}}

\newcommand{\xra}[1]{\xrightarrow{#1}}

\newcommand{\foa}{\pentagon}
\newcommand{\bsm}{\begin{smallmatrix}}
\newcommand{\esm}{\end{smallmatrix}}
\newcommand{\ra}{\rightarrow}
\begin{document}

\title[Ideal approximation in $n$-angulated categories]{Ideal approximation in $n$-angulated categories}
\author{Lingling Tan, Dingguo Wang, Tiwei Zhao$^\dagger$}
\address{School of Mathematical Sciences, Qufu Normal University, Qufu 273165, P.R. China}

\email{tanll@qfnu.edu.cn, dgwang@qfnu.edu.cn, tiweizhao@qfnu.edu.cn}

\date{}


\thanks{$^\dagger$Corresponding author}

\begin{abstract}
We study ideal approximation theory associated to almost $n$-exact structures in extension closed subcategories of $n$-angulated
categories. For $n=3$, an $n$-angulated category is nothing but a classical triangulated category. Moreover, since every exact category can be embedded as an extension closed subcategory of a triangulated category, therefore, our approach extends the recent ideal approximations theories developed by Fu, Herzog et al. for exact categories and by Breaz and Modoi for triangulated categories.
\end{abstract}

\baselineskip=15pt
\maketitle


\section{Introduction}

Recently, the study of higher homological algebra is an active topic, and
 its aim is to provide a higher version  of the classical theory. Let $n$ be an integer greater than or equal to three.
Geiss, Keller and Oppermann \cite{GKO} introduced the notion of $n$-angulated categories, which is a ``higher dimensional" analogue of triangulated categories.
Since then,  the theory of $n$-angulated categories has been developed by many authors. For example, Geiss, Keller and Oppermann gave a construction of
$n$-angulated categories in terms of certain cluster tilting subcategories.  Bergh,  Jasso and Thaule \cite{BJT} showed that for a commutative local ring with principal maximal ideal
squaring to zero, the category of finitely generated free modules is $n$-angulated for every $n\geq 3$. More generally, Lin \cite{L} presented a general framework to unify the two constructions
of $n$-angulated categories in \cite{GKO} and \cite{BJT}, and also constructed new examples of $n$-angulated
categories.
Besides, Jacobsen and J{\o}rgensen \cite{JJ}, and Zhou and Zhu \cite{ZZ} investigated the quotient categories of $n$-angulated categories. For more references, see \cite{BS,BT1,BT2,F,F2,J,L0,L2,Z}.

In an abstract category, objects and morphisms are two essential components; and by a
well-known embedding from a category to its morphism category, objects can be viewed as
special morphisms. In the classical approximation theory, we mainly concern the objects
and the associated subcategories. However, in general case, it seems that the morphisms
and the associated ideals also should be concerned in the approximation theory. From this
point of view, Fu, Guil Asensio, Herzog and Torrecillas  \cite{Fu-et-al} introduced the notion of ideal
cotorsion pairs and developed the ideal approximation theory of exact categories. Inside
it, the phantom ideal plays an important role in the aspect of providing a certain ideal
cotorsion pair. Recently, Breaz and  Modoi \cite{BM} extended ideal approximation theory into triangulated categories,
and  Asadollahi and Sadeghi \cite{AS} extended ideal approximation theory into $n$-exact categories which is a ``higher dimensional" analogue of exact categories.
The theory of ideal approximation has been developed
further, see e.g. \cite{Ch-98,Estrada-et-al,Fu-Herzog,H}.

Following these ideas, in the present paper we will study the higher dimensional analogy of Breaz and Modoi's result in \cite{BM} in extension closed subcategories of $n$-angulated categories. For $n = 3$,
an $n$-angulated category is nothing but a classical triangulated category. Moreover, since every exact category can be embedded as an extension closed subcategory
of a triangulated category, the theory presented here includes important parts from the theory developed in \cite{Fu-et-al} and \cite{BM}.

In Section \ref{Section-weak-proper}, we give some terminology and some preliminary results. We recall the definition of $n$-angulated categories, and introduce the notion of
 almost $n$-exact structures in an extension closed subcategory $\CA$ of an $n$-angulated category.

In Section \ref{section-precovering}, we study some general properties of precovers and preenvelops associated to some (phantom) ideals in $\CA$. We introduce the notions of $\frakE$-projective and $\frakE$-injective morphisms associated to an almost $n$-exact structure $\frakE$.  It is proved that the existence of enough $\frakE$-injective or $\frakE$-projective morphisms is connected with the existence of some ideal precovers or preenvelopes, see Theorem \ref{Th-procov-vs-Iinj}.
Then we define the orthogonality of morphisms associated to an almost $n$-exact structure $\frakE$, and we introduce special precovering and special preenveloping ideals.  In this context one of the main results is Salce's Lemma (Theorem \ref{salce-lemma}) which tells us that in many cases all
precovers/preenvelopes are special.
This lemma was appeared firstly in the study of approximation theory of objects, especially in the study of cotorsion theory \cite{Sa}, and  was extended to the ideal version associated to exact categories in \cite{Fu-et-al}, and associated to triangulated categories in \cite{BM}. In Theorem \ref{salce-lemma} we will prove the corresponding result in the setting of $n$-angulated categories. We also introduce the notion of (complete) ideal cotrosion pairs in terms of the orthogonality of morphisms and ideals, and  apply the previous work  to obtain a characterization for complete ideal cotorsion pairs, see Theorem \ref{cotorsion-precovering}.

The main aim of Section \ref{Section-relative-cotorsion-pairs} is to provide characterizations for complete ideal torsion pairs.
Let $\frakF$ be an almost $n$-exact structure which is included in $\frakE$. We first construct ideals from almost $n$-exact structures by introducing the notions of $\frakF$-phantoms and $\frakF$-cophantoms (relative to $\frakE$). Following them, we can get
some kinds of orthogonal pairs of ideals. On the other hand, we can reverse this process by starting with an ideal, namely   we can define two kinds of almost $n$-exact structures $\mathfrak{PB}_\frakE(\CI)$ and $\mathfrak{PO}_\frakE(\CI)$ associated to an ideal $\CI$, which are called pullback and pushout
almost $n$-exact structures respectively. In this way we obtain two Galois correspondences between the class of ideals in $\CA$ and the class of
 almost $n$-exact structures included in $\frakE$ (Theorem \ref{Galois-correspondences}).
Based on the above work, we get the main result of this paper:

\medskip

{\bf Theorem.} {\rm (Theorem \ref{mainthA})}
Let $\mathcal{T}$ be an $n$-angulated category, and $\CA$ a full subcategory of $\mathcal{T}$ which is closed under extensions.
Let $\frakE$ be an almost $n$-exact structure for $\CA$ such that there are enough $\frakE$-injective morphisms
and $\frakE$-projective morphisms, and let $(\CI,\CJ)$ be an ideal cotorsion pair in $\CA$.
Then the following statements are equivalent.
\begin{enumerate}[{\rm (a)}]
 \item $\CI$ is precovering.
 \item $\CI$ is special precovering.
 \item $\CJ$ is preenveloping.
 \item $\CJ$ is special preenveloping.
\item There exists an almost $n$-exact structure $\frakH\subseteq \frakE$ with enough special injective morphisms such that $\CI=\mathrm{\Phi}_\frakE(\frakH)$.
\item There exists an almost $n$-exact structure $\frakF\subseteq \frakE$ with enough injective morphisms such that $\CI=\PEF$.
\item There exists an almost $n$-exact structure $\frakH\subseteq \frakE$ with enough special injective morphisms such that
$\CJ=\frakH\inj$.
\item There exists an almost $n$-exact structure $\mathfrak{G}_n\subseteq \frakE$ with enough special projective morphisms such that $\CJ=\CPEG$.
\item There exists an almost $n$-exact structure $\mathfrak{F}_n\subseteq \frakE$ with enough projective morphisms such that $\CJ=\CPEF$.
\item There exists an almost $n$-exact structure $\mathfrak{G}_n\subseteq \frakE$ with enough special projective morphisms such that
$\CI=\mathfrak{G}_n\proj$.
\end{enumerate}

\section{Almost $n$-exact structures in $n$-angulated categories and ideals}\label{Section-weak-proper}

\subsection{$n$-angulated categories}

Let $\mathcal{T}$ be an additive category with an automorphism
$\Sigma$, and $n$ an integer greater or equal than three. A sequence of morphisms
in $\mathcal{T}$
\[
\mathfrak{d}: \ \ A_1\xra{\alpha_1} A_2\xra{\alpha_2}\cdots\xra{\alpha_{n-1}} A_n\xra{\alpha_n}\Sigma A_1
\]
is an $n$-$\Sigma$-\emph{sequence}. Its \emph{left rotation} is the
$n$-$\Sigma$-sequence
\[
A_2 \xra{\alpha_2} A_3 \xra{\alpha_3}\cdots\xra{\alpha_n}
\Sigma A_1\xra{(-1)^n\Sigma\alpha_1} \Sigma A_2.
\]
The $n$-$\Sigma$-sequences of the form
$A\xra{1_A} A \xra{} 0 \xra{}\cdots  \xra{}0 \xra{} \Sigma A$
for $A\in\mathcal{T}$, and their rotations, are called \emph{trivial}.

An $n$-$\Sigma$-sequence $\mathfrak{d}$ is \emph{exact} if the induced sequence
\[
\mathcal{T}(-,\mathfrak{d}):\quad\cdots\ra\mathcal{T}(-,A_1)\ra\mathcal{T}(-,A_2)\ra\cdots\ra\mathcal{T}(-,A_n)\ra
\mathcal{T}(-,\Sigma A_1)\ra\cdots
\]
 is exact. In particular, the trivial
$n$-$\Sigma$-sequences are exact.

A \emph{morphism} of $n$-$\Sigma$-sequences is given by
a sequence of morphisms $\varphi=(\varphi_1,\varphi_2,\ldots,\varphi_n)$ such
that the following diagram commutes:
\[
\xymatrix@=0.7cm{
& A_1\ar[r]^{\alpha_1}\ar[d]^{\varphi_1} & A_2\ar[r]^{\alpha_2}\ar[d]^{\varphi_2} &
A_3\ar[r]\ar[d]^{\varphi_3}&\cdots\ar[r] & A_n\ar[r]^{\alpha_n}\ar[d]^{\varphi_n} &\Sigma A_1\ar[d]^{\Sigma\varphi_1}  \\
&B_1\ar[r]^{\beta_1} &B_2\ar[r]^{\beta_2} & B_3\ar[r] &\cdots\ar[r]&B_n\ar[r]^{\beta_n} &\Sigma B_n
}
\]

We call a collection $\foa$ of $n$-$\Sigma$-sequences an
{\em $n$-angulation of  $(\mathcal{T},\Sigma)$}
and its elements $n$-\emph{angles}
if $\foa$ fulfills the following axioms:
\begin{itemize}
\item[(F1)]
  \begin{itemize}
  \item[(a)]
    $\foa$ is closed under direct sums and under taking direct summands.
  \item[(b)]
For each $A\in\mathcal{T}$, the trivial $n$-$\Sigma$ sequence
$$A\xra{1_A} A \xra{} 0 \xra{}\cdots  \xra{}0 \xra{} \Sigma A$$
  belongs to $\foa$.
  \item[(c)]
    For each morphism $\alpha_1: A_1\ra A_2$ in $\mathcal{T}$, there exists an
    $n$-angle
    \[
A_1\xra{\alpha_1} A_2\xra{}\cdots\xra{} A_n\xra{}\Sigma A_1
\]
  \end{itemize}
\item[(F2)]
An $n$-$\Sigma$-sequence $\mathfrak{d}$ belongs to $\foa$
if and only if its left rotation
\[
A_2 \xra{\alpha_2} A_3 \xra{\alpha_3}\cdots\xra{\alpha_n}
\Sigma A_1\xra{(-1)^n\Sigma\alpha_1} \Sigma A_2
\]
belongs to $\foa$.
\item[(F3)] Each commutative diagram
\[
\xymatrix@=0.7cm{
& A_1\ar[r]^{\alpha_1}\ar[d]^{\varphi_1} & A_2\ar[r]^{\alpha_2}\ar[d]^{\varphi_2} &
A_3\ar[r]&\cdots\ar[r] & A_n\ar[r]^{\alpha_n} &\Sigma A_1\ar[d]^{\Sigma\varphi_1}  \\
&B_1\ar[r]^{\beta_1} &B_2\ar[r]^{\beta_2} & B_3\ar[r] &\cdots\ar[r]&B_n\ar[r]^{\beta_n} &\Sigma B_n
}
\]
with rows in $\foa$ can be completed to a morphism of $n$-$\Sigma$-sequences
\[
\xymatrix@=0.7cm{
& A_1\ar[r]^{\alpha_1}\ar[d]^{\varphi_1} & A_2\ar[r]^{\alpha_2}\ar[d]^{\varphi_2} &
A_3\ar[r]\ar@{-->}[d]^{\varphi_3}&\cdots\ar[r] & A_n\ar[r]^{\alpha_n}\ar@{-->}[d]^{\varphi_n} &\Sigma A_1\ar[d]^{\Sigma\varphi_1}  \\
&B_1\ar[r]^{\beta_1} &B_2\ar[r]^{\beta_2} & B_3\ar[r] &\cdots\ar[r]&B_n\ar[r]^{\beta_n} &\Sigma B_n.
}
\]

\item[(F4)]
In the situation of (F3) the morphisms $\varphi_3,\varphi_4,\ldots,\varphi_n$ can be
chosen such that the cone ${\rm Cone}(\varphi)$:
\[
\xymatrix{
A_2\oplus B_1\quad
\ar[r]^{\left(\bsm -\alpha_2& 0\\ \phantom{-}\varphi_2&\beta_1\esm\right)}&
\quad A_3\oplus B_2\quad
\ar[r]^{\left(\bsm -\alpha_3& 0\\ \phantom{-}\varphi_3&\beta_2\esm\right)}&
\quad \cdots\quad
\ar[r]^{\left(\bsm -\alpha_n& 0\\ \phantom{-}\varphi_n&\beta_{n-1}\esm\right)}&
\qquad\Sigma A_1\oplus B_n\quad
\ar[r]^{\left(\bsm -\Sigma\alpha_1& 0\\ \phantom{-}\Sigma\varphi_1& \beta_n\esm\right)}&
\quad \Sigma A_2\oplus \Sigma B_1
}
\]
belongs to $\foa$.
\end{itemize}

Let $\mathcal{T}$ be an additive category with an automorphism
$\Sigma$, and $n$ an integer greater or equal than three.  The triple $(\mathcal{T},\Sigma,\foa)$ is called an $n$-\emph{angulated} category if $\foa$ is an
{ $n$-angulation of  $(\mathcal{T},\Sigma)$}. We usually say that $\mathcal{T}$ is an $n$-angulated category if there is no confusion.

\begin{remark}\label{2.1.1}
Let $\mathcal{T}$ be an $n$-angulated category.
\begin{itemize}
  \item[(1)] Given any $n$-angle $$\mathfrak{d}:A_1\overset{\alpha_1}\to A_2\overset{\alpha_2}\to \cdots\to A_n\overset{\alpha_n}\to \Sigma A_1$$
  and any morphism $g:A_1\to C_1$ in $\mathcal{T}$. By (F1)(c), we have an $n$-angle
  $$A_n\overset{\Sigma g\alpha_n}\to \Sigma C_1\to \Sigma C_2\to \cdots \to \Sigma C_{n-1}\to \Sigma A_n$$
  in $\mathcal{T}$.  Moreover, by (F3) and (F2), we can get the following morphism of $n$-angles
 \[\xymatrix@=0.5cm{           
  \mathfrak{d}:& A_1\ar[r]^{\alpha_1}\ar[d]^g & A_2\ar[r]^{\alpha_2}\ar[d]& \cdots \ar[r]& A_{n-1}\ar[r] \ar[d]   & A_n\ar[r]^{\alpha_n}  \ar@{=}[d]   & \Sigma A_1 \ar[d]^{\Sigma g}\\
  g\mathfrak{d}: & C_1\ar[r] &C_2\ar[r]& \cdots \ar[r]& C_{n-1}\ar[r]    & A_n\ar[r]^{\Sigma g \alpha_n}     & \Sigma C_1.         }\]

  Dually,  given any $n$-angle
 $$\mathfrak{d}:A_1\overset{\alpha_1}\to A_2\overset{\alpha_2}\to \cdots\to A_n\overset{\alpha_n}\to \Sigma A_1$$  and any morphism $f:B_n\to A_n$  in $\CA$. By (F1)(c) and (F2), we have an $n$-angle
  \[\xymatrix@=0.5cm{           
  \mathfrak{d}f:& A_1\ar[r] & B_2\ar[r]& \cdots \ar[r]& B_{n-1}\ar[r]    & B_n\ar[r]^{\alpha_n f}     & \Sigma A_1. }\]
 Moreover, by (F3) and (F2), we can get the following morphism of $n$-angles
 \[\xymatrix@=0.5cm{           
  \mathfrak{d}f:& A_1\ar[r]\ar@{=}[d] & B_2\ar[r]\ar[d]& \cdots \ar[r]& B_{n-1}\ar[r] \ar[d]   & B_n\ar[r]^{\alpha_n f}  \ar[d]^f   & \Sigma A_1 \ar@{=}[d]\\
  \mathfrak{d}: & A_1\ar[r]^{\alpha_1} &A_2\ar[r]^{\alpha_2}& \cdots \ar[r]& A_{n-1}\ar[r]    & A_n\ar[r]^{\alpha_n}     & \Sigma A_1.         }\]

 \item[(2)] By \cite[Proposition 2.5]{GKO}, each $n$-angle is  exact. The following statements are obvious.
 \begin{itemize}
   \item [(i)] Given any $n$-angle $$\mathfrak{d}:A_1\overset{\alpha_1}\to A_2\overset{\alpha_2}\to \cdots\to A_n\overset{\alpha_n}\to \Sigma A_1$$ in $\mathcal{T}$. Then $\alpha_{i+1}\alpha_i=0$ for any $1\leq i\leq n-1$, and $(\Sigma \alpha_1)\alpha_n=0$.
   \item [(ii)] For any $1\leq i\leq n-1$, if $f:A_i\to B$ is a morphism with $f\alpha_{i-1}=0$, then there exists a morphism $g:A_{i+1}\to B$ such that $f=g\alpha_i$, that is, there is the following commutative diagram:
      \[\xymatrix@=0.5cm{           
   A_1\ar[r]^{\alpha_1} & \cdots \ar[r]& A_{i-1}\ar[r]^{\alpha_{i-1}}\ar[rd]_0& A_{i}\ar[r]^{\alpha_i}\ar[d]^f & A_{i+1}\ar[r]^{\alpha_{i+1}}\ar@{.>}[ld]^g&\cdots\ar[r]   & A_n\ar[r]^{\alpha_n}    & \Sigma A_1 \\
     &&&B&& & }.\]
     \item [(iii)] For any $1\leq i\leq n-1$, if $f:B\to A_i$ is a morphism with $\alpha_{i}f=0$, then there exists a homomorphism $g:B\to A_{i-1}$ such that $f=\alpha_{i-1}g$, that is, there is the following commutative diagram:
      \[\xymatrix@=0.7cm{           
      &&&B\ar@{.>}[ld]_g\ar[d]^f\ar[rd]^0&& &\\
   A_1\ar[r]^{\alpha_1} & \cdots \ar[r]& A_{i-1}\ar[r]^{\alpha_{i-1}}& A_{i}\ar[r]^{\alpha_i} & A_{i+1}\ar[r]^{\alpha_{i+1}}&\cdots\ar[r]   & A_n\ar[r]^{\alpha_n}    & \Sigma A_1
     }.\]
 \end{itemize}
 \item[(3)] Following \cite[Lemma 3.1.3]{F}, for an $n$-angle $$\mathfrak{d}:A_1\overset{\alpha_1}\to A_2\overset{\alpha_2}\to \cdots\to A_{n-1}\overset{\alpha_{n-1}}\to A_n\overset{\alpha_n}\to \Sigma A_1$$ in $\mathcal{T}$, the following are equivalent:
     \begin{enumerate}[{\rm (i)}]
       \item $\alpha_n=0$.
       \item $\alpha_1$ is a split monomorphism.
       \item $\alpha_{n-1}$ is a split epimorphism.
     \end{enumerate}
     If one of the above condition holds, then we call the $n$-angle $\mathfrak{d}$ \emph{split}.
  \end{itemize}
\end{remark}

\subsection{Almost $n$-exact structures}

Let $\CT$ be an $n$-angulated category. If
$$\mathfrak{d}:A_1\overset{\alpha_1}\to A_2\overset{\alpha_2}\to \cdots\to A_n\overset{\alpha_n}\to \Sigma A_1$$ is an $n$-angle in $\CT$, we will say that $\alpha_n$ is the phantom map corresponding to
$\mathfrak{d}$.

If $\CA$ is a full subcategory (closed with respect to isomorphisms)
of $\CT$, one says that it is \textsl{closed under extensions} if for every $n$-angle
$$\mathfrak{d}:A_1\overset{\alpha_1}\to A_2\overset{\alpha_2}\to \cdots\to A_n\overset{\alpha_n}\to \Sigma A_1$$ in $\CT$
such that
 $A_1$ and $A_n$ are objects in $\CA$ then $A_2,\cdots, A_{n-1}$ are  objects in $\CA$, see \cite[Definition 3.6]{L0}. We will denote by $\frakD_\CA$ the class of all
  $n$-angles
  $$\mathfrak{d}:A_1\overset{\alpha_1}\to A_2\overset{\alpha_2}\to \cdots\to A_{n-1}\overset{\alpha_{n-1}}\to A_n\overset{\alpha_n}\to \Sigma A_1$$
   such that $A_1,A_2,\cdots, A_n\in\CA$. In this case, we will say that
\begin{itemize}
  \item $\mathfrak{d}$ is an \textsl{$n$-angle from $\CA$} or in $\frakD_\CA$.
  \item $\alpha_1$ is a $\frakD_\CA$-inflation.
  \item $\alpha_{n-1}$ is a $\frakD_\CA$-deflation.
\end{itemize}

Assume that $\mathcal{A}$ is a full subcategory of $\CT$ which is closed under extensions.
A class of $n$-angles $\frakE\subseteq \frakD_\CA$ is called an \textsl{almost $n$-exact structure} for $\CA$ if it satisfies the following conditions:
\begin{enumerate}[{\rm (N1)}]
 \item $\frakE$ is closed under
 direct sums and contains {all split $n$-angles},
 \item If
 $$\mathfrak{d}:A_1\overset{\alpha_1}\to A_2\overset{\alpha_2}\to \cdots\to A_n\overset{\alpha_n}\to \Sigma A_1$$ is an $n$-angle in $\frakE$ and $f:B_n\to A_n$ is a morphism in $\CA$, then
 the top $n$-angle $\mathfrak{d}f$ in each morphism of $n$-angles
 \[\xymatrix@=0.5cm{           
  \mathfrak{d}f:& A_1\ar[r]\ar@{=}[d] & B_2\ar[r]\ar[d]& \cdots \ar[r]& B_{n-1}\ar[r] \ar[d]   & B_n\ar[r]^{\alpha_n f}  \ar[d]^f   & \Sigma A_1 \ar@{=}[d]\\
  \mathfrak{d}: & A_1\ar[r]^{\alpha_1} &A_2\ar[r]^{\alpha_2}& \cdots \ar[r]& A_{n-1}\ar[r]    & A_n\ar[r]^{\alpha_n}     & \Sigma A_1         }\]
is in $\frakE$.
\item
 If
 $$\mathfrak{d}:A_1\overset{\alpha_1}\to A_2\overset{\alpha_2}\to \cdots\to A_n\overset{\alpha_n}\to \Sigma A_1$$ is an $n$-angle in $\frakE$ and $g:A_1\to C_1$ is a morphism in $\CA$, then
 the bottom $n$-angle $g\mathfrak{d}$ in each morphism of $n$-angles
 \[\xymatrix@=0.5cm{           
  \mathfrak{d}:& A_1\ar[r]^{\alpha_1}\ar[d]^g & A_2\ar[r]^{\alpha_2}\ar[d]& \cdots \ar[r]& A_{n-1}\ar[r] \ar[d]   & A_n\ar[r]^{\alpha_n}  \ar@{=}[d]   & \Sigma A_1 \ar[d]^{\Sigma g}\\
  g\mathfrak{d}: & C_1\ar[r] &C_2\ar[r]& \cdots \ar[r]& C_{n-1}\ar[r]    & A_n\ar[r]^{\Sigma g \alpha_n}     & \Sigma C_1         }\]
is in $\frakE$.
\end{enumerate}

Throughout this paper, unless otherwise stated, {we always assume that $\CT$ is an $n$-angulated category, $\CA$ is a full subcategory which is closed under extensions,
and $\frakE$ is an almost $n$-exact structure for $\CA$.}

Let
$$\mathfrak{d}:A_1\overset{\alpha_1}\to A_2\overset{\alpha_2}\to \cdots\to A_{n-1}\overset{\alpha_{n-1}}\to A_n\overset{\alpha_n}\to \Sigma A_1$$
be an $n$-angle which lies in  $\frakE$.
We will say that \begin{itemize} \item  $\alpha_1$ is a
{\it $\frakE$-inflation},
\item $\alpha_{n-1}$ is a {\it $\frakE$-deflation}, and
\item $\alpha_n$ is a {\it $\frakE$-phantom}.
\end{itemize}
We denoted by $\Ph(\frakE)$ the class of all $\frakE$-phantoms.

\begin{lemma}\label{infl-defl-factors}\label{lemma-factorizari-2}
Let $\frakE$ be an almost $n$-exact structure for $\mathcal{A}$, and let $f:A_1\to B_2$ and $g:B_2\to A_2$  be morphisms in $\CA$.
\begin{enumerate}[{\rm (1)}]
\item If $f$ is a $\mathfrak{D}_\CA$-inflation and $gf$ is a $\frakE$-inflation then  $f$ is a $\frakE$-inflation.
\item If $g$ is a $\mathfrak{D}_\CA$-deflation and $gf$ is a $\frakE$-deflation then  $g$ is a $\frakE$-deflation.
\end{enumerate}
\end{lemma}

\begin{proof}
We only prove (1), and (2) is similar. Since $f$ is a $\mathfrak{D}_\CA$-inflation, there is an  $n$-angles
\[\xymatrix@=0.5cm{           
   A_1\ar[r]^f & B_2\ar[r]& \cdots \ar[r]& B_{n-1}\ar[r]   & B_n\ar[r]   & \Sigma A_1 }\]
 in $\mathfrak{D}_\CA$. Since  $gf$ is a $\frakE$-inflation, there is an $n$-angle
\[\xymatrix@=0.5cm{           
    A_1\ar[r]^{gf} &A_2\ar[r]& \cdots \ar[r]& A_{n-1}\ar[r]    & A_n\ar[r]     & \Sigma A_1.         }\]
in $\frakE$ (of course, in $\mathfrak{D}_\CA$).  By (F3), there is a morphism of $n$-angles
\[\xymatrix@=0.5cm{           
   A_1\ar[r]^f\ar@{=}[d] & B_2\ar[r]\ar[d]^g& \cdots \ar[r]& B_{n-1}\ar[r] \ar[d]   & B_n\ar[r]  \ar[d]   & \Sigma A_1 \ar@{=}[d]\\
    A_1\ar[r]^{gf} &A_2\ar[r]& \cdots \ar[r]& A_{n-1}\ar[r]    & A_n\ar[r]     & \Sigma A_1.         }\]
    Finally, since $\frakE$ is an almost $n$-exact structure for $\mathcal{A}$, the top row is in $\frakE$ by (N2), and hence $f$ is a $\frakE$-inflation.
\end{proof}

\subsection{Phantom $\CA$-ideals}

For a subcategory $\CA$  of $\CT$ we denote  by $\CA^{\to}$ the class of all morphisms in $\CA$. Recall that an \textsl{ideal} $\mathcal{I}$ in $\CA$ is a class of morphisms in $\CA^\to$ which is closed under  sums of morphisms,
and for every chain
of composable  morphisms $A\overset{f}\to B\overset{i}\to C\overset{g}\to D$ in $\CA$,
if $i\in\CI$ then $gif\in\CI$, or equivalently, $\CI(-,-)$ is a subfunctor of the Hom-bifunctor $\CA(-,-)$.

Let $\mathcal{I}$ be a
class of morphisms in $\CA$. We write ${\rm Ob}(\mathcal{I}) := \{A\in\CA  \mid 1_A \in \mathcal{I}\}$. Let $\X$ be a full subcategory of $\CA$.
 Then we define
  \[<\X>=\{i\in\CA^\to\mid i\hbox{ factors through an object in }\X\}.\]
 If $\mathcal{I} = <{\rm Ob}(\mathcal{I})>$, then we call $\mathcal{I}$ an object ideal, that is,
it is generated by itself objects.

Recall from \cite[Definition 2.2.2]{BM} that  a class $\CE$ of morphisms in $\CT$ is called a \textsl{phantom $\CA$-ideal} if
\begin{enumerate}[(P1)]
\item $\CE\subseteq \CT(\CA,\Sigma\CA)=\bigcup_{A,B\in \CA}\CT(A,\Sigma B),$
\item $\CE$ is closed under  sums of morphisms,
and
\item $\Sigma\CA^{\to}\CE\CA^\to\subseteq \CE$, i.e. for every chain
of composable  morphisms $$A\overset{f}\to B\overset{i}\to \Sigma C\overset{\Sigma g}\to \Sigma D$$ in $\CT$
such that $i\in\CE$ and $f,g\in\CA^\to$ we have $(\Sigma g)if\in\CE$.
\end{enumerate}

In general, phantom $\CA$-ideals and ideals in $\CA$ are different notions. However,
if $\Sigma\CA=\CA$, in particular for $\CA=\CT$, then  a class of morphisms $\CI$ is a phantom $\CA$-ideal if and only if it is an ideal in $\CA$.

It is easy to check that $\Ph(\frakE)$ is a phantom $\CA$-ideal.  In fact,
there is an 1-to-1 correspondence between almost $n$-exact structures and
phantom $\CA$-ideals as follows. 

\begin{proposition}\label{pseudo-ph-vs-wpc}
Let $\frakE$ be a class of $n$-angles in $\frakD_\CA$. %
Then the following are equivalent.
\begin{enumerate}[{\rm (a)}]
 \item $\frakE$ is an almost $n$-exact structure for $\CA$;
 \item $\Sigma\CA^{\to}\Ph(\frakE)\CA^\to\subseteq \Ph(\frakE)$ and $\Ph(\frakE)$ is closed under (direct) sums of morphisms.
\end{enumerate}

Consequently, \begin{enumerate}[{\rm (i)}]
               \item
If $\frakE$ is an almost $n$-exact structure for $\CA$, then $\Ph(\frakE)$ is a phantom $\CA$-ideal,
and
\item for every phantom $\CA$-ideal $\CI$ the class $$\frakD_\CA(\CI):=\{\mathfrak{d}\in\frakD_\CA\mid \hbox{the phantom of }\mathfrak{d}\hbox{ is in }\CI\}$$
is an almost $n$-exact structure for $\CA$.
\item The correspondences from {\rm (i)} and {\rm (ii)} above are inverse to each other.
\end{enumerate}
\end{proposition}

\begin{example}
Let $\mathcal{B}$ be a full subcategory in $\CT$
and let $\frakE$ be an almost $n$-exact structure for $\CA$.
Then $\CB$ induces
an almost $n$-exact structure $\mathfrak{F}_\CB\subseteq \frakE$ defined by the condition
$$\Ph(\mathfrak{F}_\CB)=\{\varphi\in\Ph(\frakE)\mid \varphi \text{ factorizes through an object } X\in\CB\}.$$
\end{example}

Let $\frakE$ be an almost $n$-exact structure for $\CA$.
Given two $n$-angles
\[\xymatrix@=0.5cm{           
 \mathfrak{d}: & A_1\ar[r] & B_2\ar[r]& \cdots \ar[r]& B_{n-1}\ar[r]   & A_n\ar[r]^{\alpha_n}    & \Sigma A_1 }\]
and
\[\xymatrix@=0.5cm{           
\mathfrak{d'}:& A_1\ar[r] &A_2\ar[r]& \cdots \ar[r]& A_{n-1}\ar[r]    & A_n\ar[r]^{\alpha_n}     & \Sigma A_1.         }\]
in $\frakE$.
We write $\mathfrak{d}\rightsquigarrow\mathfrak{d'}$  if there is a morphism of $n$-angles of the form:
\[\xymatrix@=0.5cm{           
 \mathfrak{d}: & A_1\ar[r]\ar@{=}[d] & B_2\ar[r]\ar[d]^{f_2}& \cdots \ar[r]& B_{n-1}\ar[r] \ar[d]^{f_{n-1}}   & A_n\ar[r]  \ar@{=}[d]   & \Sigma A_1 \ar@{=}[d]\\
\mathfrak{d'}:& A_1\ar[r] &A_2\ar[r]& \cdots \ar[r]& A_{n-1}\ar[r]    & A_n\ar[r]     & \Sigma A_1.         }\]
The relation $\rightsquigarrow$ satisfies the properties of reflexivity and transitivity, and thus generates  an equivalence relation $\thicksim$ on the class of all $n$-angles in $\frakE$
starting in $A_1$ and ending in $A_n$.

Since two $n$-angles in $\frakE$ are equivalent if and only if they have the same $\frakE$-phantom, it follows that the class of all $n$-angles in $\frakE$ starting in $A_1$ and ending in $A_n$ is a set modulo the equivalence of $n$-angles. We denote by $\frakE(A_n,A_1)$ this set.

By assigning to each $n$-angle in $\frakE$ its $\frakE$-phantom map we have an isomorphism
\[\frakE(A_n,A)\cong\Ph(\frakE)(A_n,\Sigma A).\]

\begin{remark}
Let $\frakE$ be an almost $n$-exact structure for $\CA$ and let
$$\mathfrak{d}:A_1\overset{\alpha_1}\to A_2\overset{\alpha_2}\to \cdots\to A_n\overset{\alpha_n}\to \Sigma A_1$$
 be an $n$-angle in $\frakE$.
If $f:X\to A_n$ and $g:A_1\to Y$ are two morphisms, we can construct the following morphisms of $n$-angles:
\[ \xymatrix@=0.5cm{ \mathfrak{d}: & A_1\ar[r]\ar@{=}[d] & A_2\ar[r]\ar@{<-}[d]&\cdots\ar[r]& A_{n-1} \ar[r]\ar@{<-}[d]  & A_n \ar[r]^{\phi}\ar@{<-}[d]^{f} & \Sigma A_1\ar@{=}[d]\\
\mathfrak{d}f: & A_1\ar[r]\ar[d]^{g} & B_2\ar[r]\ar[d] &\cdots\ar[r]& B_{n-1} \ar[r]\ar[d] &X \ar[r]^{\phi f}\ar@{=}[d] & \Sigma A_1\ar[d]^{\Sigma g} \\
g(\mathfrak{d}f): & Y\ar[r] & C_2\ar[r]&\cdots\ar[r] & C_{n-1} \ar[r]& X \ar[r]^{(\Sigma g)\phi f} & \Sigma C_1
}\]
We can also construct the following morphisms of $n$-angles:
 \[ \xymatrix@=0.5cm{ \mathfrak{d}: & A_1\ar[r]\ar[d]^g & A_2\ar[r]\ar[d]&\cdots\ar[r]& A_{n-1} \ar[r]\ar[d]  & A_n \ar[r]^{\phi}\ar@{=}[d] & \Sigma A_1\ar[d]^{\Sigma g}\\
g\mathfrak{d}: & Y\ar[r]\ar@{=}[d] & B'_2\ar[r] &\cdots\ar[r]& B'_{n-1} \ar[r] &A_n\ar[r]^{(\Sigma g)\phi} & \Sigma Y\ar@{=}[d] \\
(g\mathfrak{d})f: & Y\ar[r] & C'_2\ar[u]\ar[r]&\cdots\ar[r] & C'_{n-1} \ar[u]\ar[r]& X\ar[u]^f \ar[r]^{(\Sigma g)\phi f} & \Sigma C_1
}\]
Since both $n$-angles $g(\mathfrak{d}f)$ and $(g\mathfrak{d})f$ have the same $\frakE$-phantom map, namely $(\Sigma g)\alpha_n f$, it follows that $g(\mathfrak{d}f)$ and $(g\mathfrak{d})f$  are equivalent.
\end{remark}

An almost $n$-exact structure $\frakE$ is called an {\it $n$-exact structure} provided that it satisfies one of the equivalent conditions in the following lemma.

\begin{lemma}
Let $\frakE$ be an almost $n$-exact structure for $\CA$.
The following statements are equivalent.

\begin{enumerate}[{\rm (a)}]
\item If $A_1,C_1,A_n\in\CA$, $i:A_1\to C_1$ is a $\frakE$-inflation and $\alpha_n:A_n\to \Sigma A_1$, then $(\Sigma i)\alpha_n\in\Ph(\frakE)$ implies $\alpha_n\in\Ph(\frakE)$.
\item  Given any commutative diagram
\[\xymatrix@=0.5cm{           
  \mathfrak{d}:& A_1\ar[r]^{\alpha_1}\ar[d]^i & A_2\ar[r]^{\alpha_2}\ar[d]& \cdots \ar[r]& A_{n-1}\ar[r] \ar[d]   & A_n\ar[r]^{\alpha_n}  \ar@{=}[d]   & \Sigma A_1 \ar[d]^{\Sigma i}\\
 i\mathfrak{d}: & C_1\ar[r] &C_2\ar[r]& \cdots \ar[r]& C_{n-1}\ar[r]    & A_n\ar[r]^{(\Sigma i) \alpha_n}     & \Sigma C_1         }\]
   of $n$-angles in $\mathfrak{D}_\CA$ with $i$ a $\frakE$-inflation.
Then  $i\mathfrak{d}\in\frakE$ implies  $\mathfrak{d}\in\frakE$.
\item If $A_1,B_n,A_n\in\CA$, $p:B_n\to A_n$ is a $\frakE$-deflation and $\phi:A\to \Sigma C$, then $\alpha_n p\in\Ph(\frakE)$ implies $\alpha_n\in\Ph(\frakE)$.
\item Given any commutative diagram
 \[\xymatrix@=0.5cm{           
  \mathfrak{d}p:& A_1\ar[r]\ar@{=}[d] & B_2\ar[r]\ar[d]& \cdots \ar[r]& B_{n-1}\ar[r] \ar[d]   & B_n\ar[r]^{\alpha_n p}  \ar[d]^p   & \Sigma A_1 \ar@{=}[d]\\
  \mathfrak{d}: & A_1\ar[r]^{\alpha_1} &A_2\ar[r]^{\alpha_2}& \cdots \ar[r]& A_{n-1}\ar[r]    & A_n\ar[r]^{\alpha_n}     & \Sigma A_1         }\]
   of $n$-angles in $\mathfrak{D}_\CA$ with $p$ a $\frakE$-deflation.
Then  $\mathfrak{d}p\in\frakE$ implies $\mathfrak{d}\in \frakE$.
\end{enumerate}
\end{lemma}

\begin{proof}
The equivalences (a)$\Leftrightarrow$(b) and (c)$\Leftrightarrow$(d) are obvious. Moreover, (a)$\Rightarrow$(c) and (c)$\Rightarrow$(a) are dual to each other, so we only prove (a)$\Rightarrow$(c).

{ Let $p:B_n\to A_n$ be a $\frakE$-deflation and let $\alpha_n:A_n\to \Sigma A_1$ be a map such that $A_1,B_n,A_n\in\CA$ and $\alpha_n p\in\Ph(\frakE)$. Completing both $p$ and $\alpha_n p$ to $n$-angles we obtain the following commutative diagram:
\[\xymatrix{
 & A_2\ar[r]\ar[d] &\cdots\ar[r]&A_{n-1}\ar[r]\ar[d]& B_n\ar@{=}[d]\ar[r]^{p}    & A_n\ar[r]^{\psi}\ar[d]^{\alpha_n}   & \Sigma A_2\ar[d]\\
A_1\ar[r]^{i} & B_2\ar[r]&\cdots\ar[r]&B_{n-1}\ar[r]           & B_n\ar[r]^{\alpha_n p}    & \Sigma A_1\ar[r]^{-\Sigma i} &\Sigma B_2          \\
}\]
By hypothesis, $\psi\in\Ph(\frakE)$, $i$ is an $\frakE$-inflation and $(\Sigma i)\alpha_n\in\Ph(\frakE)$. Then (a) implies
$\alpha_n\in\Ph(\frakE)$.
}
\end{proof}

\section{Precovering and preenveloping ideals}\label{section-precovering}

\subsection{Precovers and preenvelopes}

Let $\CI$ be an ideal in $\CA$, and $A$ an object in $\CA$.

A morphism $i:X\to A$ is called an {\it $\CI$-precover} of $A$
if $i\in\CI$ and for any other morphism $i':X'\to A$ in $\CI$, there exists a morphism $g:X'\to X$ such that
the following diagram
\[
\xymatrix{&X'\ar@{.>}[ld]_g\ar[d]^{i'}\\
X\ar[r]^i&A}
\]
commutes.

The notion of an {\it $\CI$-preenvelope} for an object is given dually.

The ideal $\CI$ is a {\it precovering} (respectively, {\it preenveloping}) ideal
if every object from $\CA$ has an $\CI$-precover (respectively, $\CI$-preenvelope).

Since the suspension functor is an equivalence, for every $n\in\Z$,
$i:X\to A$ is an $\CI$-precover of $A$ if and only if $\Sigma^ni:\Sigma^nX\to \Sigma^nA$ is a $\Sigma^n\CI$-precover of
$\Sigma^nA$. Similarly, $j:B\to Y$ is an $\CI$-preenvelope of $B$ if and only if $\Sigma^nj:\Sigma^nB\to \Sigma^nY$ is a $\Sigma^n\CI$-preenvelope of
$\Sigma^nB$.

We extend these notions for phantom $\CA$-ideals in the following way: if $\CE$ is a phantom $\CA$-ideal and $A\in\CA$,
we say that a morphism $\phi:X\to \Sigma A$ is an \textsl{$\CE$-precover} for $\Sigma A$
if $\phi\in\CE$ and all morphisms $\phi':X'\to \Sigma A$ in $\CE$ factorize through $\phi$. Dually,
an \textsl{$\CE$-preenvelope} for an object $B$ in $\CA$ is a morphism
$\phi:B\to \Sigma Y$ which lies in $\CE$ such that every other morphism $\phi':B\to \Sigma Y'$ from $\CE$
factorizes through $\phi$. The phantom
$\CA$-ideal $\CE$ is \textsl{precovering} (respectively, \textsl{preenveloping})
if every object from $\Sigma \CA$ (respectively, $\CA$) has an $\CE$-precover (respectively, $\CE$-preenvelope).

Let $\frakE$ be an almost $n$-exact structure for $\CA$.
We say that a morphism $f:X\to A$ in $\CA$ is
\textsl{$\frakE$-projective} if $f$ is projective with respect to  all $n$-angles in $\frakE$,
that is, $f$ factorizes through all $\frakE$-deflations $A_{n-1}\to A$.
Dually, $g:A\to Y$ is \textsl{$\frakE$-injective} in $\CA$ if $g$ is injective with respect to  all $n$-angles in $\frakE$, that is,
$g$ factorizes through all $\frakE$-inflations $A\to A_2$.
We denote by $\Eproj$ ($\Einj$) the class of all $\frakE$-projective
(respectively, $\frakE$-injective) morphisms.

\begin{lemma}\label{basic-CI-inj}
Let $\frakE$ be an almost $n$-exact structure for $\CA$.
\begin{enumerate}[{\rm (a)}]
 \item A morphism $f:X\to A$ in $\CA$ is $\frakE$-projective if and only if $\Ph(\frakE) f=0$.
 \item  A morphism $g:A\to Y$ in $\CA$ is $\frakE$-injective if and only if $(\Sigma g) \Ph(\frakE)=0$.
 \item $\Eproj$ and $\Einj$ are ideals in $\CA$.
 \item $\Sigma^m\frakE\text{-}\mathrm{proj}=\Sigma^m(\Eproj)$ and $\Sigma^m\frakE\text{-}\mathrm{inj}=\Sigma^m(\Einj)$ for all $m\in\Z$,
 where $\Sigma^m\frakE$ is viewed as an almost $n$-exact structure relative to the full subcategory $\Sigma^m\CA$.
\end{enumerate}
\end{lemma}

\begin{proof}
We only prove (a), (b) is dual, and (c) and (d) are easy.

The if part of (a).
Assume  $\Ph(\frakE) f=0$. For any $\frakE$-deflation $\alpha_{n-1}:A_{n-1}\to A$, we choose an $n$-angle
\[
\xymatrix@=0.5cm{A_1\ar[r]&\cdots\ar[r]&A_{n-1}\ar[r]^{\alpha_{n-1}}&A\ar[r]^\alpha&\Sigma A_1}
\]
in $\frakE$. In particular, $\alpha\in\Ph(\frakE)$, and hence
 $\alpha f=0$ by assumption. Thus $f$ factors through $\alpha_{n-1}$  by Remark \ref{2.1.1}(2)(iii).

 The only if part of (a). Assume that $f:X\to A$ is $\frakE$-projective. For any $\alpha: A\to \Sigma A_1\in\Ph(\frakE)$,
  we choose an $n$-angle
\[
\xymatrix@=0.5cm{A_1\ar[r]&\cdots\ar[r]&A_{n-1}\ar[r]^{\alpha_{n-1}}&A\ar[r]^-\alpha&\Sigma A_1}
\]
in $\frakE$. By the definition of $\frakE$-projective morphisms, there exists a morphism $g:X\to A_{n-1}$ such that the following diagram
\[
\xymatrix@=0.7cm{&&&X\ar@{.>}[ld]_g\ar[d]^f&\\
A_1\ar[r]&\cdots\ar[r]&A_{n-1}\ar[r]^-{\alpha_{n-1}}&A\ar[r]^-\alpha&\Sigma A_1}
\]
is commutative.
By Remark \ref{2.1.1}(2)(i), $\alpha f=\alpha \alpha_{n-1} g=0$. Thus $\Ph(\frakE) f=0$.
\end{proof}

The above mentioned connection is presented in the following results.

\begin{lemma}\label{Icov-vs-Iinjenv} Let $\frakE$ be an almost $n$-exact structure  for $\CA$, and let
  $$\mathfrak{d}:A_1\overset{\alpha_1}\to A_2\overset{\alpha_2}\to \cdots\to A_{n-1}\overset{\alpha_{n-1}}\to A_n\overset{\alpha_n}\to \Sigma A_1$$
   be an $n$-angle in $\frakE$.
  \begin{enumerate}[{\rm (1)}]
\item If  $\alpha_n$ is a $\Ph(\frakE)$-precover for $\Sigma A_1$ and $d:X\to A_2$ is a morphism
 such that $\alpha_2d=0$, then $d\in\Einj$. In particular $\alpha_1$ is a $\Einj$-preenvelope of $A_1$.

 As a consequence, the following statements are equivalent:
 \begin{enumerate}[{\rm (a)}]
   \item $\alpha_n$ is a $\Ph(\frakE)$-precover for $\Sigma A_1$.
   \item  $\alpha_1$ is a $\Einj$-preenvelope of $A_1$.
   \item $\alpha_1$ is $\frakE$-injective.
 \end{enumerate}

\item If  $\alpha_n$ is a $\Ph(\frakE)$-preenvelope for $A_n$ and $d:A_{n-1}\to X$ is a morphism such that
$d\alpha_{n-2}=0$, then $d\in\Eproj$. In particular $\alpha_{n-1}$ is a $\Eproj$-precover for $A_n$.

As a consequence, the following statements are equivalent:
 \begin{enumerate}[{\rm (a)}]
   \item $\alpha_n$ is a $\Ph(\frakE)$--preenvelope for $A_n$.
   \item  $\alpha_{n-1}$ is a $\Eproj$-precover for $A_n$.
   \item $\alpha_{n-1}$ is $\frakE$-projective.
 \end{enumerate}
\end{enumerate}
\end{lemma}

\begin{proof} We only prove (1), and (2) is dual.

 Let $\psi:Y\to \Sigma X$ be a morphism in $\Ph(\frakE)$. 
Since $\alpha_2d=0$, by Remark \ref{2.1.1}(2)(ii) there exists a morphism $h:X\to A_1$ such that  $d=\alpha_1h$. Now $(\Sigma h)\psi:Y\to \Sigma A_1$ is in $\Ph(\frakE)$
 because $\Ph(\frakE)$ is a phantom $\CA$-ideal. Since $\alpha_n$
 is a $\Ph(\frakE)$-precover for $\Sigma A_1$, 
 there exists a morphism $k:Y\to A_n$ such that  $(\Sigma h)\psi=\alpha_n k$. That is, we obtain the following commutative diagram:
 \[ \xymatrix{ & X\ar@{-->}[dl]_{h}\ar[d]^{d} && & & Y \ar@{-->}[dl]_{k}\ar[r]^{\psi} & \Sigma X\ar@{-->}[dl]_{\Sigma h}\ar[d]^{\Sigma d}   \\
 A_1 \ar[r]^{\alpha_1}& A_2\ar[r]^{\alpha_2}&\cdots\ar[r] & A_{n-1}\ar[r]^{\alpha_{n-1}} & A_n \ar[r]^{\alpha_n} & \Sigma A_1\ar[r]^{\Sigma \alpha_1} & \Sigma A_2
 .} \]
 Now $(\Sigma d)\psi=(\Sigma \alpha_1)(\Sigma h)\psi=(\Sigma \alpha_1)\alpha_n k=0$ by Remark \ref{2.1.1}(2)(i). Therefore $(\Sigma d)\Ph(\frakE)=0$, and so $d\in\Einj$ by Lemma \ref{basic-CI-inj}(2).

In particular,  since $\alpha_2\alpha_1=0$, $\alpha_1$ is $\frakE$-injective by the first assertion.
 Moreover, if
 $f':A_1\to B'$ is a morphism in $\Einj$, by the definition of $\frakE$-injective morphisms $f'$ has to factor through $\alpha_1$. Thus $\alpha_1$ is a $\Einj$-preenvelope of $A_1$.

Finally, if $\alpha_1\in\Einj$, then for every map $\psi:X\to \Sigma A_1$ from $\Ph(\frakE)$ we
have $(\Sigma \alpha_1)\psi=0$  by Lemma \ref{basic-CI-inj}(2). Hence $\psi$ factors through $\alpha_n$ by Remark \ref{2.1.1}(2)(ii), and thus $\alpha_n$ is a $\Ph(\frakE)$-precover for $\Sigma A_1$.
 \end{proof}

\begin{definition}\label{def-suficiente}
Let $\frakE$ be an almost $n$-exact structure for $\CA$. We say that there are \textsl{enough $\frakE$-injective morphisms}
if for every object $A$ there
 exists a $\frakE$-inflation $f:A\to E$ which is $\frakE$-injective.

Dually, there are \textsl{enough $\frakE$-projective morphisms} if for every object $C$ there
 exists a $\frakE$-deflation $g:P\to C$ which is $\frakE$-projective.
\end{definition}

Now we show that the existence of enough $\frakE$-injective (respectively, $\frakE$-projective) morphisms is
equivalent to the precovering (respectively, preenveloping) property of the phantom $\CA$-ideal $\Ph(\frakE)$, and it recover \cite[Theorem 3.1.5]{BM} in case $n=3$.

\begin{theorem}\label{Th-procov-vs-Iinj}
Let $\frakE$ be an almost $n$-exact structure for $\CA$.
	\begin{enumerate}[{\rm (1)}]
\item The following are equivalent: \begin{enumerate}[{\rm (a)}]
\item  there are enough $\frakE$-injective morphisms;
 \item $\Ph(\frakE)$ is a precovering phantom $\CA$-ideal.

 \end{enumerate}
\item The following are equivalent:  \begin{enumerate}[{\rm (a)}] \item there are enough $\frakE$-projective morphisms;
 \item $\Ph(\frakE)$ is a preenveloping phantom $\CA$-ideal.
 \end{enumerate}
\end{enumerate}
\end{theorem}

\begin{proof}

(a)$\Rightarrow$(b) Let $A\in\CA$. By (a), there exists an $n$-angle
$$A\overset{f}\to A_2\to \cdots \to A_{n}\overset{\phi}\to \Sigma A$$ in $\frakE$ such that $f$ is $\frakE$-injective. By  Lemma \ref{Icov-vs-Iinjenv}(1),
 $\phi$ is a
$\Ph(\frakE)$-precover for $\Sigma A$.

(b)$\Rightarrow$(a)
Suppose that $\Ph(\frakE)$ is a precovering phantom $\CA$-ideal. Let $A\in\CA$.
If $\phi:A_n\to \Sigma A$ is a $\Ph(\frakE)$-precover, we choose an $n$-angle
$A\overset{f}\to A_2\to \cdots \to A_n\overset{\phi}\to \Sigma A$ in $\frakE$. Using Lemma \ref{Icov-vs-Iinjenv}(1), we
conclude that $f$ is a $\frakE$-injective $\frakE$-inflation.
\end{proof}

In the following we will present a method to construct almost $n$-exact structures with enough injective/projective morphisms.
We start with an construction of  almost $n$-exact structures for $\CA$.

Let $\CI$ be an ideal in $\CA$. Then $$\mathcal{E}^{\CI}=\{\varphi\in\CT(\CA,\Sigma \CA)\mid (\Sigma i)\varphi=0\mbox{ for any }i\in\CI\}$$ is a phantom $\CA$-ideal,
hence the class $\frakE^\CI:=\frakD_{\CA}(\mathcal{E}^{\CI})$ is an almost $n$-exact structure for $\CA$ by Proposition \ref{pseudo-ph-vs-wpc}.
It is easy to see that $\frakE^\CI$ consists  of all $n$-angles
$A_1\to A_2\to\cdots \to A_n\overset{\varphi}\to \Sigma A_1$ from $\mathfrak{D}_\CA$ with the property that all morphisms from $\CI$ are $\frakE^\CI$-injective, that is,
 $\CI\subseteq \frakE^\CI\mbox{-inj}$.

Dually, if we consider the phantom $\CA$-ideal $$\mathcal{E}_{\CI}=\{\varphi\in\CT(\CA,\Sigma \CA)\mid \varphi i=0\mbox{ for any }i\in\CI\},$$
we obtain an almost $n$-exact structure $\frakE_\CI:= \frakD_{\CA}(\mathcal{E}_{\CI})$ consisting of all $n$-angles
$A_1\to A_2\to\cdots \to A_n\overset{\varphi}\to \Sigma A_1$ from $\mathfrak{D}_\CA$ with the property that $\CI\subseteq\frakE_\CI\mbox{-proj}$.

\begin{proposition}\label{prop1-suficiente}
Let $\frakE$ be an almost $n$-exact structure for $\CA$.
\begin{enumerate}[{\rm (1)}]
\item If there are enough $\frakE$-injective morphisms, then $\frakE=\frakE^{\frakE\textrm{-}\mathrm{inj}}$.
\item If there are enough $\frakE$-projective morphisms, then $\frakE=\frakE_{\frakE\textrm{-}\mathrm{proj}}$.
\end{enumerate}
\end{proposition}

\begin{proof} We only prove (1), and (2) is dual.

 Take $\CI:=\frakE\textrm{-inj}$. By Lemma \ref{basic-CI-inj}(b), $\frakE\subseteq\frakE^\CI$. Thus it is enough to prove the inclusion $\frakE^\CI\subseteq \frakE$.

Let $$\xymatrix@=0.5cm{           
  \mathfrak{d}:& A_1\ar[r] & A_2\ar[r]& \cdots \ar[r]& A_{n-1}\ar[r]   & A_n\ar[r]  & \Sigma A_1   }$$
   be an $n$-angle in $\frakE^\CI$. Let $\alpha:A_1\to B_2$ be a $\frakE$-injective $\frakE$-inflation. Since $\alpha\in\CI\subseteq \frakE^\CI\mbox{-inj}$,
 we can construct a commutative diagram
\[\xymatrix@=0.5cm{           
  A_1\ar[r]\ar@{=}[d] & A_2\ar[r]\ar[d]& \cdots \ar[r]& A_{n-1}\ar[r] \ar[d]   & A_n\ar[r]  \ar[d]   & \Sigma A_1 \ar@{=}[d]\\
   A_1\ar[r]^{\alpha} &B_2\ar[r]& \cdots \ar[r]& B_{n-1}\ar[r]    & B_n\ar[r]     & \Sigma A_1         }\]
where the horizontal lines are $n$-angles in $\mathfrak{D}_\CA$. Since the below $n$-angle is in $\frakE$, the top one is also in $\frakE$ by the condition (N2). Thus $\frakE^\CI\subseteq \frakE$.
\end{proof}

It is easy to see that if $\frakF$ and $\frakE$ are almost $n$-exact structures for $\CA$ such that $\frakF\subseteq \frakE$, then $\frakE\inj\subseteq
\frakF\inj$ and $\frakE\proj\subseteq \frakF\proj$. We can use the previous proposition to prove a converse for this implication.

\begin{corollary}\label{cor-prop1-suficiente}
Let $\frakE$ and $\frakF$ be almost $n$-exact structures for $\CA$.
\begin{enumerate}[{\rm (1)}]
\item Suppose that there are enough $\frakE$-injective morphisms. Then $\frakF\subseteq \frakE$
if and only if $\frakE\inj\subseteq \frakF\inj$.
\item Suppose that there are enough $\frakE$-projective morphisms. Then $\frakF\subseteq \frakE$
if and only if $\frakE\proj\subseteq \frakF\proj$.
\end{enumerate}
\end{corollary}

\begin{proof}
We only prove (1).  Suppose $\frakE\inj\subseteq \frakF\inj$. Since $\Sigma(\frakF\inj)\Ph(\frakF)=0$, we have $\Ph(\frakF)\subseteq \mathcal{E}^{\frakF\inj}$,
which implies that $\frakF= \frakD_{\CA}(\Ph(\frakF))\subseteq\frakD_{\CA}(\mathcal{E}^{\frakF\inj})=\frakE^{\frakF\inj}$. Then  $\frakF\subseteq \frakE^{\frakF\inj}\subseteq \frakE^{\frakE\inj}=\frakE$ by Proposition \ref{prop1-suficiente}.
\end{proof}

The following proposition shows that we can construct an almost $n$-exact structure with enough injective (respectively, projective) morphisms in terms of
a certain preenveloping (respectively, precovering) ideal, which extends \cite[Proposition 3.1.8]{BM}.

\begin{proposition}\label{subclasses-enough-inj}
Let $\frakF\subseteq \frakE$ be almost $n$-exact structures for $\CA$.
\begin{enumerate}[{\rm (1)}]
\item If there are enough $\frakE$-injective morphisms, then the following are equivalent:  \begin{enumerate}[{\rm (a)}]
\item there are enough $\frakF$-injective morphisms;
\item  there exists a preenveloping ideal $\CI$ in $\CA$ such that $\frakE\textrm{-}\mathrm{inj}\subseteq \CI$, and $\frakF=\frakE^\CI$.
\end{enumerate}

\noindent Under these conditions, $\frakF\inj=\CI$.
\item If there are enough $\frakE$-projective morphisms, then the following are equivalent: \begin{enumerate}[{\rm (a)}]
\item there are enough $\frakF$-projective morphisms;
\item  there exists a precovering ideal $\CI$ in $\CA$ such that $\frakE\textrm{-}\mathrm{proj}\subseteq \CI$ and $\frakF=\frakE_\CI$.
 \end{enumerate}

\noindent Under these conditions, $\frakF\proj=\CI$.
\end{enumerate}
\end{proposition}

\begin{proof} We only prove (1).

(a)$\Rightarrow$(b) Take $\CI=\frakF\textrm{-}\mathrm{inj}$. The conclusion follows from Definition \ref{def-suficiente} and Proposition
\ref{prop1-suficiente}.

(b)$\Rightarrow$(a) Suppose that $\CI$ is a preenveloping ideal in $\CA$ such that $\frakE\textrm{-}\mathrm{inj}\subseteq \CI$, and $\frakF=\frakE^\CI$. We will prove that there are  enough $\frakE^\CI$-injective morphisms. Let $A\in\CA$.
We start with an $n$-angle
$$\xymatrix@=0.5cm{           
 & A\ar[r]^e & E\ar[r]&B_3\ar[r]& \cdots \ar[r]& B_{n-1}\ar[r]   & B_n\ar[r]  & \Sigma A   }$$
  in $\frakE$
such that $e$ is $\frakE$-injective. Let $i:A\to I$ be an
$\CI$-preenvelope for $A$. Since $e\in \frakE\textrm{-}\mathrm{inj}\subseteq\CI$, there exists a morphism $g:I\to E$ such that $e=gi$. Choose an $n$-angle
$$\xymatrix@=0.5cm{           
 & A\ar[r]^i & I\ar[r]&A_3\ar[r]& \cdots \ar[r]& A_{n-1}\ar[r]   & A_n\ar[r]  & \Sigma A.   }$$
Then we have a morphism of $n$-angles as follows:
\[\xymatrix@=0.5cm{           
  A\ar[r]^i\ar@{=}[d] & I\ar[r]\ar[d]^g& A_3\ar[r]\ar[d]&\cdots \ar[r]& A_{n-1}\ar[r] \ar[d]   & A_n\ar[r]  \ar[d]   & \Sigma A \ar@{=}[d]\\
   A\ar[r]^{e} &E\ar[r]&B_3\ar[r]&  \cdots \ar[r]& B_{n-1}\ar[r]    & B_n\ar[r]     & \Sigma A         }\]
By the condition (N2), we know that the top $n$-angle is in $\frakE$, and hence  $i$ is an $\frakE$-inflation.
Moreover, since $i$ is an $\CI$-preenvelope, for any $i':A\to I'$, there exists a morphism $g:I\to I'$ such that $i'=gi$, which implies that
 $(\Sigma i')\varphi=(\Sigma g)(\Sigma i)\varphi=0 $. Thus the $n$-angle $$\xymatrix@=0.5cm{           
 & A\ar[r]^i & I\ar[r]&A_3\ar[r]& \cdots \ar[r]& A_{n-1}\ar[r]   & A_n\ar[r]^{\varphi}  & \Sigma A   }$$
  is in $\frakE^\CI$, and hence
$i$ is a $\frakE^\CI$-injective $\frakE^\CI$-inflation.

For the last statement, let us observe that for every $A\in \CA$ every $\CI$-preenvelope $i:A\to I$ is a $\frakF$-inflation.
Therefore, every $\frakF$-injective morphism $A\to X$ factorizes through $i$,
hence $\frakF\inj\subseteq \CI$ since $\CI$ is an ideal. On the other hand, $\CI\subseteq \frakE^\CI\inj=\frakF\inj$ since $\frakF=\frakE^\CI$.
Thus $\frakF\inj=\CI$.
\end{proof}

{

\begin{example}\label{ex-constr-suf-proj}
If $\CA=\CT$ and $\frakE=\foa$ is the class consisting of all $n$-angles in $\CT$ then $\foa\mbox{-inj}=0$ and $\foa\mbox{-proj}=0$. Since in this case all morphisms are $\foa$-inflations and $\foa$-deflations, it follows that there are enough
$\foa$-injective morphisms and $\foa$-projective morphisms.

If $\CI$ is a preenveloping (respectively, precovering) ideal, we consider the almost $n$-exact structure $\foa^\CI$ (respectively, $\foa_\CI$) of all $n$-angles $\mathfrak{d}$ such that all $i\in \CI$ are injective (respectively, projective) relative to $\mathfrak{d}$. By what we just proved we obtain that $\foa^\CI$
(respectively, $\foa_\CI$) has enough injective (respectively, projective) morphisms and $\foa^\CI\inj=\CI$ (respectively, $\foa_\CI\proj=\CI$).
\end{example}
}

\subsection{Orthogonality}

\begin{definition}
We say that a morphism $f:X\to A$ from $\CA$ is \textsl{left orthogonal}
(with respect to $\frakE$)
to a morphism $g:B\to Y$ from $\CA$,
and we denote this by $f\perp g$, 
if $$\CT(f,\Sigma g)(\Ph(\frakE))=0,$$ i.e.
$$(\Sigma g)\phi f=0$$
for all morphisms $\phi:A\to \Sigma B$ in $\Ph(\frakE)$
\end{definition}

\begin{remark}
Let $f:X\to A$ and $g:B\to Y$ be morphisms in $\CA$. The following are equivalent:
\begin{enumerate}[{\rm (a)}]
  \item $f\perp g$.
  \item for every $n$-angle \[\xymatrix@=0.5cm{ A_1\ar[r] & A_2\ar[r]&\cdots\ar[r]& A_{n-1} \ar[r]  & A_n \ar[r]^{\phi} & \Sigma A_1}\] in $\frakE$, the bottom $n$-angle obtained in the following diagrams
 \[ \xymatrix@=0.5cm{ A_1\ar[r]\ar@{=}[d] & A_2\ar[r]\ar@{<-}[d]&\cdots\ar[r]& A_{n-1} \ar[r]\ar@{<-}[d]  & A_n \ar[r]^{\phi}\ar@{<-}[d]^{f} & \Sigma A_1\ar@{=}[d]\\
A_1\ar[r]\ar[d]^{g} & B_2\ar[r]\ar[d] &\cdots\ar[r]& B_{n-1} \ar[r]\ar[d] & B_n \ar[r]^{\phi f}\ar@{=}[d] & \Sigma A_1\ar[d]^{\Sigma g} \\
C_1\ar[r] & C_2\ar[r]&\cdots\ar[r] & C_{n-1} \ar[r]& B_n \ar[r]^{(\Sigma g)\phi f} & \Sigma C_1
}\]
is split.
\end{enumerate}
\end{remark}

\begin{example}
If $\CA=\CT$ and
$\frakE=\foa$, then $f\perp g$ if and only if $\CT(f,\Sigma g)=0$.
\end{example}

\begin{lemma}\label{lema-ort-basic}
Let $f:A_{n-1}\to A_n$ and $g:B_1\to  B_2$ be two morphisms in $\CA$. The following are equivalent:
\begin{enumerate}[{\rm (1)}]
 \item $f\perp g$;
 \item every morphism $\phi:A_n\to \Sigma B_1$ in $\Ph(\frakE) $ induces a morphism of $n$-angles
 \[ \xymatrix@=0.5cm{ A_1 \ar[r] \ar@{-->}[d]&A_2\ar[r]\ar@{-->}[d]&\cdots\ar[r] & A_{n-1}\ar[r]^f \ar@{-->}[d] & A_n \ar[r] \ar[d]^{\phi} & \Sigma A_1\ar@{-->}[d]  \\
 B_2 \ar[r]& B_2\ar[r]&\cdots\ar[r]  & B_n\ar[r]&\Sigma B_1 \ar[r]_{\Sigma g}  & \Sigma B_2
 .}
 \]
 \end{enumerate}
\end{lemma}

\begin{proof}
Let $\phi:A\to \Sigma B$ be a morphism in $\Ph(\frakE) $. If we complete $f$ and $g$ to $n$-angles above, respectively below, we obtain a
diagram \[ \xymatrix@=0.5cm{ A_1 \ar[r] &A_2\ar[r]&\cdots\ar[r] & A_{n-1}\ar[r]^f  & A_n \ar[r] \ar[d]^{\phi} & \Sigma A_1  \\
 B_2 \ar[r]& B_2\ar[r]&\cdots\ar[r]  & B_n\ar[r]&\Sigma B_1 \ar[r]_{\Sigma g}  & \Sigma B_2
 .}
 \]
Therefore, $f\perp g$ if and only if $(\Sigma g)\phi f=0$, if and only if there exists a morphism $A_{n-1}\to B_n$ such that the square
\[ \xymatrix@=0.5cm{  A_{n-1}\ar[r]^f \ar@{-->}[d] & A_n  \ar[d]^{\phi}   \\
  B_n\ar[r]  & \Sigma B_1
 }
 \]
 is commutative.
 \end{proof}

 Let $\CM$ be a class of maps in $\CA$. We define
 \begin{align*}
    \CM^{\perp}&=\{g\in\CA^\to \mid m\perp g \text{ for all } m\in \CM\} \\
    & =\{g\in\CA^\to \mid (\Sigma g)\phi m=0 \text{ for all } m\in \CM\mbox{ and for all } \phi\in \Ph(\frakE)\},
 \end{align*}
 and
\begin{align*}
    {^{\perp}\CM}&=\{g\in\CA^\to \mid g\perp m \text{ for all } m\in \CM\} \\
    & =\{g\in\CA^\to \mid (\Sigma m)\phi g=0 \text{ for all } m\in \CM\mbox{ and for all } \phi\in \Ph(\frakE)\}.
 \end{align*}

  The proof of the next lemma is straightforward:

 \begin{lemma}
  Let $\CM$ be a class of morphisms in $\CA$. Then
  \begin{enumerate}[{\rm (1)}]
 \item $\CM^{\perp}$ and ${^{\perp}\CM}$ are ideals in $\CA$.
 \item $\Sigma^m\CM^\perp=\left(\Sigma^m\CM\right)^\perp$ and $\Sigma^m{^\perp}{\CM}={^\perp}{\left(\Sigma^m\CM\right)}$ for all $m\in\Z$,
 where the ideals $\left(\Sigma^m\CM\right)^\perp$ and ${^\perp}{\left(\Sigma^m\CM\right)}$ of $\Sigma^m\CA$ are computed with respect to
 $\Sigma^m\frakE$.
\end{enumerate}
   \end{lemma}

\subsection{Special precovers and special preenvelopes}

If $\CI$ is an ideal in $\CA$, a morphism $i:X\to A$ in $\CI$ is a \textsl{special $\CI$-precover (\wrt $\frakE$)} of $A$
if in the corresponding $n$-angle $$\xymatrix@=0.5cm{           
  \mathfrak{d}:& A_1\ar[r] & \cdots \ar[r]&A_{n-2}\ar[r]& X\ar[r]^i   & A\ar[r]^-{\psi}  & \Sigma A_1   }$$
   we have $\psi\in (\Sigma\CI^\perp) \Ph(\frakE)$,
i.e. $\psi=(\Sigma j)\phi$ for some $j\in\CI^\perp$ and some $\phi\in\Ph(\frakE)$.

Note that $j\in\CI^\perp$ means $(\Sigma j)\varphi i=0$ for all $\varphi\in\Ph(\frakE)$ and all $i\in \CI$.

We say that $\CI$ is a \textsl{special precovering ideal} if every object  in $\CA$ has a special $\CI$-precover.

Dually, if $\CJ$ is an ideal in $\CA$, a morphism $j:B\to Y$ in $\CJ$ is a \textsl{special $\CJ$-preenvelope (\wrt $\frakE$)} of $B$
if in the corresponding $n$-angle $$\xymatrix@=0.5cm{           
  \mathfrak{d}:& B\ar[r]^j&Y\ar[r]&B_3 \ar[r]& \cdots \ar[r]& B_n\ar[r]^\psi   &  \Sigma B   }$$ we have
$\psi\in \Ph(\frakE)\,({^\perp{ \CJ}})$, i.e. $\psi=\phi i$ with $i\in {^\perp{ \CJ}}$,
$\phi\in\Ph(\frakE)$.

Note that $i\in {^\perp{ \CJ}}$ means $(\Sigma j)\varphi i=0$ for all $j\in\CJ$ and all $\varphi\in\Ph(\frakE)$.

We say that $\CJ$ is a \textsl{special preenveloping ideal} if every object  in $\CA$ has a special $\CJ$-preenvelope.

\begin{remark}
A morphism $i:X\to A$ is a {special $\CI$-precover} of $A$ if and only if
there exists a morphism of $n$-angles
\[{\rm (SPC)}\xymatrix@=0.5cm{           
  \mathfrak{d}:& B_1\ar[r]\ar[d]_j & B_2\ar[r]\ar[d]& \cdots \ar[r]& B_{n-2}\ar[r] \ar[d]& B_{n-1}\ar[r] \ar[d]   & A\ar[r]^-\phi  \ar@{=}[d]   & \Sigma B_1 \ar[d]^{\Sigma j}\\
  j\mathfrak{d}: & A_1\ar[r] &A_2\ar[r]& \cdots \ar[r]& A_{n-2}\ar[r] & X\ar[r]^i    & A\ar[r]^-{\psi}     & \Sigma A_1         }\]
such that $j\in \CI^{\perp}$, and the top $n$-angle $\mathfrak{d}\in\frakE$.

Dually, a morphism $j:B\to Y$ is a {special $\CJ$-preenvelope} of $B$
if and only if there exists  a morphism of $n$-angles
\[{\rm (SPE)}\xymatrix@=0.5cm{           
  \mathfrak{d}i:& B\ar[r]^j\ar@{=}[d] & Y\ar[r]\ar[d]&B_3\ar[r]\ar[d]& \cdots \ar[r]& B_{n-1}\ar[r] \ar[d]   & B_n\ar[r]^-{\psi}  \ar[d]^i  & \Sigma B \ar@{=}[d]\\
  \mathfrak{d}: & B\ar[r] &A_2\ar[r]& A_3\ar[r]&\cdots \ar[r]& A_{n-1}\ar[r]    & A_n\ar[r]^-\phi     & \Sigma B         }\]
such that $i\in {^{\perp}}\CJ$, and the bottom $n$-angle $\mathfrak{d}\in\frakE$.
\end{remark}

In both diagrams (SPC) and (SPE), since $\phi\in \Ph(\frakE)$, $\psi\in\Ph(\frakE) $, then all horizontal $n$-angles are in $\frakE$
. This shows that every special $\CI$-precover (respectively, $\CJ$-preenvelope)
is a $\frakE$-deflation (respectively, $\frakE$-inflation).

Moreover, since $\psi=(\Sigma j) \phi$ and $j\in \CI^{\perp}$, we have $\psi \CI=0$ in (SPC). Similarly, $(\Sigma\CJ) \psi=0$ in (SPE).

\begin{lemma}\label{3.3.2}
Let $\CI$ and $\CJ$ be ideals  in $\CA$.
\begin{enumerate}[{\rm (1)}]
\item Every special $\CI$-precover is an $\CI$-precover.
\item Every special $\CJ$-preenvelope is a $\CJ$-preenvelope.
\end{enumerate}
\end{lemma}

\begin{proof} We only prove (1). Let $i:X\to A$ be a {special $\CI$-precover} of $A$. If $i':X\to A$ is a map in $\CI$ then in (SPC) we have $\psi i'=(\Sigma j)\phi i'=0$. Consequently
$i'$ has to factor through $i$.
\end{proof}

The role of special precovers and special preenvelopes is exhibited by the following
version of Salce's Lemma, which is a higher dimensional version for \cite[Theorem 18]{Fu-et-al} and \cite[Theorem 3.3.3]{BM}.

\begin{theorem}\label{salce-lemma} (Salce's Lemma)
Let $\CI$ and $\CJ$ be ideals in $\CA$.
 \begin{enumerate}[{\rm (1)}]
  \item If there are enough $\frakE$-injective morphisms and $\CI$ is a precovering ideal,
 then
     $\CI^{\perp}$ is a special preenveloping ideal.
    \item If there are enough $\frakE$-projective morphisms and $\CJ$ is a preenveloping ideal,
 then
  ${^{\perp}\CJ}$ is a special precovering ideal.
 \end{enumerate}
\end{theorem}

\begin{proof} We only prove (1).

Let $A\in\CA$ and choose an $n$-angle \[\tag{IE} A\overset{e}\to E\to A\to \cdots\to A_n\overset{\psi}\to \Sigma A\]
such that $e$ is an $\frakE$-inflation which is $\frakE$-injective.

Since $\CI$ is a precovering ideal, there exists an $\CI$-precover $i:I\to A_n$. Consider a morphism of $n$-angles
\[\xymatrix@=0.5cm{
   A\ar[r]^{a}\ar@{=}[d] & J\ar[d]\ar[r] &B_3\ar[r]\ar[d]&\cdots\ar[r]   & I\ar[d]^{i} \ar[r]^-{\psi i}&\Sigma A\ar@{=}[d] \\
   A\ar[r]^{e}           & E\ar[r] &A_3\ar[r]  &\cdots\ar[r] & A_n         \ar[r]^-{\psi}&\Sigma A.
}\]
We claim that $a$ is a special $\CI^{\perp}$-preenvelope of $A$.
In order to prove this, it is enough to show that $a\in\CI^{\perp}$ since from the obvious inclusion $\CI\subseteq{^{\perp}(\CI^{\perp})}$  we know that $i\in {^{\perp}(\CI^{\perp})}$.

Since $\Sigma$ is an equivalence, in the subcategory $\Sigma^{-1}\CA$, the morphism $\Sigma^{-1}i:\Sigma^{-1}I\to \Sigma^{-1}A_n$ is an $\Sigma^{-1}\CI$-precover of $\Sigma^{-1}A_n$, and we have
the solid part of the following commutative diagram:
\[\xymatrix@=0.6cm{ \Sigma^{-1}Y\ar@{-->}[r]^{\Sigma^{-1}\kappa} \ar@{..>}[d]^{\eta} & \Sigma^{-1}B \ar@{-->}[d]^{\Sigma^{-1}\varphi} \ar@{..>}[ddl]_<<<<<<<{\zeta} & &  &&&& \\
  \Sigma^{-1}I\ar[r]_{\ \ \ \ \ \ u}\ar[d]_{\Sigma^{-1}i} & A\ar[r]^{a}\ar@{=}[d] & J\ar[d]\ar[r]    &B_3\ar[r]\ar[d]&\cdots\ar[r]   & I\ar[d]^{i} \ar[r]^-{\psi i}&\Sigma A\ar@{=}[d]\\
  \Sigma^{-1}A_n\ar[r]^{(-1)^n\Sigma^{-1}\psi} & A\ar[r]^{e}         & E\ar[r] &A_3\ar[r]  &\cdots\ar[r] & A_n         \ar[r]^-{\psi}&\Sigma A.
}\]

Let $\kappa:Y\to B$ be in $\CI$ and let $\varphi:B\to \Sigma A$ be in $\Phi_n(\frakE)$. Since $e$ is $\frakE$-injective
we have $e(\Sigma^{-1}\varphi)=0$, hence there exists a morphism $\zeta:\Sigma^{-1}B\to \Sigma^{-1}A_n$ such that $\Sigma^{-1}\varphi=((-1)^n\Sigma^{-1}\psi)\zeta$. Since $\Sigma^{-1}\CI$
is an ideal in $\Sigma^{-1}\CA$, we have $\zeta(\Sigma^{-1}\kappa)\in\Sigma^{-1}\CI$, hence $\zeta(\Sigma^{-1}\kappa)$ factorizes through the $\Sigma^{-1}\CI$-precover $\Sigma^{-1}i$, that is, there exists a morphism $\eta:\Sigma^{-1}Y\to \Sigma^{-1}I$
such that $\zeta(\Sigma^{-1}\kappa)=(\Sigma^{-1}i)\eta$.
Finally, $$a(\Sigma^{-1}\varphi)(\Sigma^{-1}\kappa)=(-1)^na(\Sigma^{-1}\psi)\zeta(\Sigma^{-1}\kappa)=
(-1)^na(\Sigma^{-1}\psi)(\Sigma^{-1} i) \eta=au\eta=0,$$
hence $(\Sigma a)\varphi\kappa=0$. Therefore, $a\in\CI^{\perp}$.
\end{proof}

In the proof of Theorem \ref{salce-lemma} we actually obtain the following corollary.

\begin{corollary}\label{salce-lemma-cor}
{\rm (1)} If $\CI$ is an ideal in $\CA$ and \[\xymatrix@=0.5cm{
   A\ar[r]^{a}\ar@{=}[d] & J\ar[d]\ar[r] &B_3\ar[r]\ar[d]&\cdots\ar[r]   & I\ar[d]^{i} \ar[r]^-{\psi i}&\Sigma A\ar@{=}[d] \\
   A\ar[r]^{e}           & E\ar[r] &A_3\ar[r]  &\cdots\ar[r] & A_n         \ar[r]^-{\psi}&\Sigma A.
}\] is a commutative diagram in $\CA$ such that the horizontal lines are $n$-angles in $\frakE$, the morphism $e$ is
$\frakE$-injective and $i$ is an $\CI$-precover of $A_n$, then the morphism $a$ is a special
$\CI^\perp$-preenvelope of $A$.

{\rm (2)} If $\CJ$ is an ideal in $\CA$  and \[\xymatrix@=0.6cm{
 A_1\ar[d]^{j}\ar[r] &\ar[r]A_2\ar[r]\ar[d]&\cdots\ar[r]   & A_{n-1}\ar[d] \ar[r]^p& A\ar@{=}[d] \ar[r] & \Sigma A_1 \ar[d]^{\Sigma j} \\
 J\ar[r]  &B_2\ar[r]&\cdots\ar[r]  & Y          \ar[r]^{b} & A \ar[r] & \Sigma J
}\] is a commutative diagram in $\CA$ such that the horizontal lines are $n$-angles in $\frakE$, $p$ is
$\frakE$-projective  and $j$ is a $\CJ$-preenvelope of $A_1$, then $b$ is a special
${^\perp\CJ}$-precover of $A$.
\end{corollary}

\subsection{Ideal cotorsion pairs}

\begin{definition}
Let $\CI,\CJ$ be ideals in $\CA$.
\begin{enumerate}[{\rm (1)}]
  \item The pair $(\CI,\CJ)$  is said to be \textsl{orthogonal} (with respect to  $\frakE$) if $i\perp j$ for all $i\in \CI$ and $j\in \CJ$,
i.e. $\CJ\subseteq \CI^\perp$ and $\CI\subseteq {^\perp \CJ}$.
  \item An \textsl{ideal cotorsion pair} (with respect to  $\frakE$) is a pair of ideals
$(\CI, \CJ)$ in $\CA$ such that $\CJ=\CI^\perp$ and $\CI=\, {^\perp \CJ}$.
\item The ideal cotorsion pair $(\CI,\CJ)$ is \textsl{complete} if $\CI$ is a special precovering ideal and $\CJ$ is a special
preenveloping ideal.
\end{enumerate}
\end{definition}

Now using the terminology of complete ideal cotorsion pairs, we can give another expression of Salce's Lemma, which is just \cite[Theorem 3.4.1]{BM} in case
$n=3$.

\begin{theorem}\label{cotorsion-precovering}
Let $\CI,\CJ$ be ideals in $\CA$.
\begin{enumerate}[{\rm (1)}]
  \item If $\CI$ is a special precovering ideal, then $(\CI,\CI^{\perp})$ is an ideal cotorsion pair.
Moreover, if there are enough $\frakE$-injective morphisms, then the ideal cotorsion pair
$(\CI,\CI^\perp)$ is complete.
  \item If $\CJ$ is a special preenveloping ideal, then $({^\perp \CJ},\CJ)$ is an ideal cotorsion pair.
Moreover, if there are enough $\frakE$-projective morphisms, then the ideal cotorsion pair
$({^\perp \CJ},\CJ)$ is complete.
\end{enumerate}
\end{theorem}

\begin{proof}
We have to show that $\CI={^{\perp}(\CI^{\perp})}$. The inclusion $\CI\subseteq {^{\perp}(\CI^{\perp})}$
is obvious.

Let $i':X'\to A$ be in ${^{\perp}(\CI^{\perp})}$. Since $\CI$ is a special precovering ideal,
we can choose an $n$-angle $$\xymatrix@=0.5cm{           
  \mathfrak{d}:& A_1\ar[r] & \cdots \ar[r]&A_{n-2}\ar[r]& X\ar[r]^i   & A\ar[r]^-{\psi}  & \Sigma A_1   }$$ such that $i$ is a special $\CI$-precover for $A$.
Then $\psi= (\Sigma j)\phi$ for some $j\in \CI^{\perp}$ and some $\phi\in\Ph(\frakE)$.
All these data are represented in the solid part of the following commutative
diagram: \[\xymatrix@=0.5cm{           
  B_1\ar[r]\ar[d]_j & B_2\ar[r]\ar[d]& \cdots \ar[r]& B_{n-2}\ar[r] \ar[d]& B_{n-1}\ar[r] \ar[d]   & A\ar[r]^-\phi  \ar@{=}[d]   & \Sigma B_1 \ar[d]^{\Sigma j}\\
   A_1\ar[r] &A_2\ar[r]& \cdots \ar[r]& A_{n-2}\ar[r] & X\ar[r]^i    & A\ar[r]^-{\psi}     & \Sigma A_1        \\
   &&&&&X'\ar[u]_{i'}\ar@{-->}[ul]^g&}\]

Because $i'\perp j$ we obtain
$\psi i'=(\Sigma j)\phi i'=0$, thus there exists a morphism $g:X'\to X$ such that $i'=ig$ by Remark \ref{2.1.1}(2)(iii).
Hence $i'\in\CI$, and therefore,  ${^{\perp}(\CI^{\perp})}\subseteq\CI$.

The second statement follows from Theorem \ref{salce-lemma}.
\end{proof}

{
\begin{example}\label{trivial-case-special}\label{ex-trivial}
If $\CA=\CT$ and $\frakE=\foa$, then every precovering ideal $\CI$ is special since
every $n$-angle $$\xymatrix@=0.5cm{           
  \mathfrak{d}:& A_1\ar[r] & \cdots \ar[r]&A_{n-2}\ar[r]& X\ar[r]^i   & A\ar[r]^-{\psi}  & \Sigma A_1   }$$
   can be embedded in a commutative diagram
\[\xymatrix@=0.5cm{           
  \Sigma^{-1}A\ar[r]\ar[d]_{\Sigma^{-1}\psi} & 0\ar[r]\ar[d]& \cdots \ar[r]& 0\ar[r] \ar[d]& 0\ar[r] \ar[d]   & A\ar[r]^{1_A}  \ar@{=}[d]   & A \ar[d]^{\psi}\\
   A_1\ar[r] &A_2\ar[r]& \cdots \ar[r]& A_{n-2}\ar[r] & X\ar[r]^i    & A\ar[r]^-{\psi}     & \Sigma A_1        }\]
   Indeed, assume that $i$ is an $\CI$-precover of $A$. For any $\phi:B\to A\in\CT$ and any $i': X'\to B\in\CI$, since $\CI$ is an ideal, $\phi i'\in\CI$, and hence there exists $h:X'\to X$ such that $ih=\phi i'$, and hence $\psi \phi i'=\psi i h=0$. Thus $\Sigma^{-1}\psi\in\CI^\perp$ which shows that
   $i$ is a special $\CI$-precover of $A$.

Dually, every preenveloping ideal in $\CT$ is special.

Therefore, for every precovering ideal $\CI$  we obtain that $(\CI,\CI^\perp)$ is a complete ideal cotorsion pair,
hence $\CI^\perp$ is
a special preenveloping ideal.
\end{example}

\section{Ideal cotorsion pairs and relative phantom ideals}\label{Section-relative-cotorsion-pairs}

In this section we extend the ideal cotorsion theory introduced in \cite{Fu-et-al} and \cite{BM} to $n$-angulated categories.

\subsection{Relative phantom ideals}
Given two almost $n$-exact structures $\frakF\subseteq\frakE$, we will construct the ideal of  phantom morphisms of $\frakF$ relative to $\frakE$.

\begin{definition} Let $\frakF$ be an almost $n$-exact structure for $\CA$ such that $\frakF\subseteq\frakE$.
A morphism $\phi:X\to A$ in $\CA$ is called an \textsl{$\frakF$-phantom relative to $\frakE$}, if $h\phi\in\Ph(\frakF)$,
whenever $h\in\Ph(\frakE) $. We denote by $$\PEF=\{\phi\mid h\phi\in\Ph(\frakF)\hbox{ for all }h\in\Ph(\frakE)\}$$
 the class of all
$\frakF$-phantoms  relative to $\frakE$.

Dually, a morphism $\psi:A\to X$ in $\CA$ is called an \textsl{$\frakF$-cophantom relative to $\frakE$}, if $\psi h\in\Sigma^{-1}\Ph(\frakF)$,
whenever $h\in\Sigma^{-1}\Ph(\frakE) $. We denote by $$\CPEF=\{\psi\mid \psi h\in\Sigma^{-1}\Ph(\frakF)\hbox{ for all }h\in\Sigma^{-1}\Ph(\frakE)\}$$
the class of all
$\frakF$-cophantoms  relative to $\frakE$.
\end{definition}

 The proof of the following lemma is straightforward.

\begin{lemma}
If $\frakF$ is an almost $n$-exact structure for $\CA$ and $\frakF\subseteq \frakE$, then $\PEF$ and $\CPEF$ are ideals in $\CA$.
\end{lemma}

\begin{remark}
\begin{enumerate}[{\rm (a)}]
  \item A morphism $\phi:X\to A$ belongs to $\PEF$ if and only if for every  morphism of $n$-angles
 in $\frakE$ along $\phi$
\[\xymatrix@=0.5cm{           
  \mathfrak{d}i:& A_1\ar[r]\ar@{=}[d] & B_2\ar[r]\ar[d]&B_3\ar[r]\ar[d]& \cdots \ar[r]& B_{n-1}\ar[r] \ar[d]   & X\ar[r]^-{h\phi}  \ar[d]^{\phi}  & \Sigma A_1 \ar@{=}[d]\\
  \mathfrak{d}: & A_1\ar[r] &A_2\ar[r]& A_3\ar[r]&\cdots \ar[r]& A_{n-1}\ar[r]    & A\ar[r]^-h     & \Sigma A_1         }\]
the top $n$-angle $\mathfrak{d}i$ is in $\frakF$.
  \item A morphism $\psi:A\to X$ belongs to $\CPEF$ if and only if for every morphism of $n$-angles
 in $\frakE$ along $\psi$
\[\xymatrix@=0.6cm{
\mathfrak{d}:& A\ar[d]^{\psi}\ar[r] &\ar[r]A_2\ar[r]\ar[d]&\cdots\ar[r]   & A_{n-1}\ar[d] \ar[r]& A_{n}\ar@{=}[d] \ar[r]^{\Sigma h} & \Sigma A \ar[d]^{\Sigma \psi} \\
\psi \mathfrak{d}:& X\ar[r]  &B_2\ar[r]&\cdots\ar[r]  & B_{n-1}          \ar[r] & A_n \ar[r]^{\Sigma (\psi h)} & \Sigma X
}\]
the bottom $n$-angle $\psi \mathfrak{d}$ is in $\frakF$.
\end{enumerate}
\end{remark}

Remark that $\PEF$ and $\CPEF$ are ideals in $\CA$ whereas $\Ph(\frakE)$ is a
phantom $\CA$-ideal.

\begin{theorem}\label{orth-F}
Let $\frakF$ be an almost $n$-exact structure for $\CA$ such that $\frakF\subseteq \frakE$. \begin{enumerate}[{\rm (1)}]
\item \begin{enumerate}[{\rm (a)}] \item
The pair $(\Phi_\frakE(\frakF), \frakF\textrm{-}\mathrm{inj})$ is orthogonal.

\item Assume that there are enough $\frakF$-injective morphisms. Then ${^{\perp} \frakF\textrm{-}\mathrm{inj}}=
\PEF$.
\end{enumerate}
\item \begin{enumerate}[{\rm (a)}] \item
The pair $(\frakF\textrm{-}\mathrm{proj}, \CPEF)$ is orthogonal.

\item Asssume that there are enough $\frakF$-projective morphisms. Then ${ \frakF\textrm{-}\mathrm{proj}^{\perp}}=
\CPEF$.
\end{enumerate}
\end{enumerate}
\end{theorem}

\begin{proof}
We only prove (1).
(a) Let $e:B\to Y$ be an $\frakF$-injective morphism, $f:X\to A\in\Phi_\frakE(\frakF)$, and $\varphi:A\to \Sigma B\in\Ph(\frakE)$.
We have $\varphi f\in\Ph(\frakF)$, hence $(\Sigma e)\varphi f=0$ since $e$ is $\frakF$-injective. Then $f\perp e$.

(b) By (a), it is enough to show that
${^{\perp} \frakF\textrm{-}\mathrm{inj}}\subseteq \Phi_\frakE(\frakF)$.
In order to do this, let  $f:X\to A\in{^{\perp}\frakF\textrm{-}\mathrm{inj}}$ and
$\varphi:A\to \Sigma B\in\Ph(\frakE)$. Let $e:B\to C$ be an $\frakF$-injective $\frakF$-inflation.

Consider the following commutative
diagram
\[ \xymatrix@C=0.7cm@R=0.5cm{ B_1\ar[r]\ar@{=}[d] & A_2\ar[r]\ar@{<-}[d]&\cdots\ar[r]& A_{n-1} \ar[r]\ar@{<-}[d]  & A \ar[r]^-{\varphi}\ar@{<-}[d]^{f} & \Sigma B_1\ar@{=}[d]\\
B_1\ar[r]^i\ar[d]_{e} & B_2\ar[r]\ar[d]^{e_2}\ar@{-->}[dl]_z &\cdots\ar[r]& B_{n-1} \ar[r]\ar[d] & B \ar[r]^-{\varphi f}\ar@{=}[d] & \Sigma B_1\ar[d]^{\Sigma g} \\
C\ar[r]^{c_1}\ar@/_/ @{<--}[r]_{c_1'}  & C_2\ar[r]&\cdots\ar[r] & C_{n-1} \ar[r]& B \ar[r]^{(\Sigma e)\varphi f} & \Sigma C
}\]
Since $(\Sigma e)\varphi f=0$, it follows that the bottom $n$-angle splits. By Remark \ref{2.1.1}(3), there exists a morphism $c_1':C_2\to C$ such that
$c_1'c_1=1_C$. Setting $z:=c_1'e_2$. Then  $zi=c_1'e_2 i=c_1'c_1e=e$.
Since $e$ is an $\frakF$-inflation, by Lemma \ref{infl-defl-factors} it follows that $i$ is an $\frakF$-inflation,
hence $\varphi f\in\Ph(\frakF)$.
Therefore,  $f\in \Phi_\frakE(\frakF)$, and hence ${^{\perp} \frakF\textrm{-}\mathrm{inj}}\subseteq \Phi_\frakE(\frakF)$.
\end{proof}



The following result shows us that $\frakE$-projective $\frakE$-deflations (respectively, $\frakE$-injective $\frakE$-inlations)
are test maps for $\frakF$-phantoms relative to $\frakE$ (respectively, $\frakF$-cophantoms  relative to $\frakE$).

\begin{proposition}\label{rel-phantom-vs-proj}
{\rm (1)} Let $$A_1\to \cdots\to A_{n-2}\to P\overset{p}\to A\overset{\psi}\to \Sigma A_1$$ be an $n$-angle in $\frakE$ such that $p$ is $\frakE$-projective.
A morphism $\varphi:X\to A$ belongs to $\Phi_\frakE(\frakF)$ 
if and only if $\psi\varphi$ belongs to $\Ph(\frakF)$.

{\rm (2)} Let $$ A\overset{e}\to E\to A_3\to \cdots\to A_n\overset{\psi}\to \Sigma A$$ be an $n$-angle in $\frakE$ such that
$e$ is $\frakE$-injective.
A morphism $\varphi:A\to Y$ belongs to $\Phi^\frakE(\frakF)$ if and only if $(\Sigma \varphi)\psi$ belongs to $\Ph(\frakF)$.
\end{proposition}

\begin{proof} We only prove (1).
Suppose that $\psi\varphi\in\Ph(\frakF)$.
We have to show that  $\zeta\varphi\in \Ph(\frakF)$ for every morphism $\zeta:A\to \Sigma B$ in $\Ph(\frakE)$.

Let $\zeta:A\to \Sigma B$ be in $\Ph(\frakE)$.
Since $p$ is $\frakE$-projective, we have $\zeta p=0$ by Lemma \ref{basic-CI-inj},
hence there exists a morphism $\Sigma g:\Sigma A_1\to \Sigma B$ in $\Sigma\CA$ such that $(\Sigma g)\psi=\zeta$.  Moreover, we have $\psi\varphi\in\Ph(\frakF)$,
and it follows that $\zeta\varphi=(\Sigma g)\psi\varphi\in\Ph(\frakF)$, hence $\varphi\in\Phi_\frakE(\frakF)$.

The converse is obvious by applying the definition of $\Phi_\frakE(\frakF)$.
\end{proof}

}

\subsection{The pullback and pushout almost $n$-exact stuctures}
In this subsection we intend to construct almost $n$-exact structures which are included in $\frakE$  from ideals in $\CA$.

The \textsl{pullback almost $n$-exact structure associated to an ideal $\CI$} in $\CA$ is defined to be the  class
\begin{align*}
  \mathfrak{PB}_\frakE(\CI)= & \left\{\mathfrak{d}\in\frakE\mid \mbox{the $\frakE$-phantom } \phi \mbox{ of } \mathfrak{d} \mbox{ can be written as }\phi=\psi i \right. \\
   & \left.\mbox{ with }i\in\mathcal{I} \mbox{ and }\psi\in\Ph(\frakE)\right\},
\end{align*}
the corresonding phantom $\CA$-ideal is $\PB_\frakE(\CI):=\Ph(\mathfrak{PB}_\frakE(\CI))=\Ph(\frakE)\CI$.

Dually, we define \textsl{the pushout almost $n$-exact structure associated to an ideal $\CJ$} of $\CA$
 \begin{align*}
  \mathfrak{PO}_\frakE(\CI)= & \left\{\mathfrak{d}\in\frakE\mid \mbox{the $\frakE$-phantom } \phi \mbox{ of } \mathfrak{d} \mbox{ can be written as }\phi= (\Sigma j)\psi \right. \\
   & \left.\mbox{ with }j\in\mathcal{J} \mbox{ and }\psi\in\Ph(\frakE)\right\},
\end{align*}
 the corresponding phantom $\CA$-ideal
$\PO_\frakE(\CJ):=\Ph(\mathfrak{PO}_\frakE(\CI))=(\Sigma\CJ)\Ph(\frakE)$.

\begin{proposition}\label{fantome-prin-proiective}
Let $\frakF$ be an almost $n$-exact structure for $\CA$ such that $\frakF\subseteq \frakE$, and let $\CI$ be an ideal in $\CA$.
\begin{enumerate}[{\rm (1)}]
\item Assume that there exist enough $\frakF$-projective morphisms. Then the following are equivalent:
\begin{enumerate}[{\rm (a)}]
\item $\CI\subseteq \PEF$;
\item $\CI(\frakF\proj)\subseteq \frakE\proj$.
\end{enumerate}
\item Assume that there exist enough $\frakF$-injective homomorphisms. Then the following are equivalent:
\begin{enumerate}[{\rm (a)}]
\item $\CI\subseteq\CPEF$;
\item $(\frakF\inj)\CI\subseteq \frakE\inj$.
\end{enumerate}
\end{enumerate}
\end{proposition}

\begin{proof} We only prove (1).

(a)$\Rightarrow$(b) Assume $\CI\subseteq \PEF$. Then $\Ph(\frakE)\CI\subseteq \Ph(\frakE)\PEF\subseteq \Ph(\frakF)$. Thus $\Ph(\frakE)\CI(\frakF\proj)\subseteq \Ph(\frakF)(\frakF\proj)=0$,
 and hence $\CI(\frakF\proj)\subseteq \frakE\proj$ by Lemma \ref{basic-CI-inj}.

(b)$\Rightarrow$(a) Assume $\CI(\frakF\proj)\subseteq \frakE\proj$. Then
$${\PB}_\frakE(\CI)(\frakF\proj)= \Ph(\frakE) \CI (\frakF\proj)\subseteq \Ph(\frakE) (\frakE\proj)=0,$$
thus $\frakF\proj\subseteq \mathfrak{PB}_\frakE(\CI)\proj$. By Corollary \ref{cor-prop1-suficiente},
$\mathfrak{PB}_\frakE(\CI)\subseteq \frakF$.
Now let $i\in\mathcal{I}$. For any $\psi\in \Ph(\frakE)$, $\psi i\in\Ph(\frakE)\CI=\Ph(\mathfrak{PB}_\frakE(\CI))\subseteq \Ph(\frakF)$, which means $i\in  \PEF$.
Thus $\CI\subseteq \PEF$.
\end{proof}

We denote by $\mathbf{Ideals}(\CA)$  the class
 of all ideals in $\CA$, and by $\mathbf{Ex}(\frakE)$ the class of all  almost $n$-exact structures included in $\frakE$.
 We will construct two Galois correspondences between $\mathbf{Ideals}(\CA)$ and $\mathbf{Ex}(\frakE)$
for $\CA$.

\begin{theorem}\label{Galois-correspondences}

 The pairs of correspondences %
  $$\mathfrak{PB}_\frakE:\mathbf{Ideals}(\CA)\rightleftarrows \mathbf{Ex}(\frakE):{\Phi}_\frakE,$$
 respectively
 $$\mathfrak{PO}_\frakE:\mathbf{Ideals}(\CA)\rightleftarrows \mathbf{Ex}(\frakE):{\Phi}^\frakE,$$
 determine two monotone Galois connections with respect to inclusion.
 \end{theorem}

\begin{proof}
Let $\CI\in \mathbf{Ideals}(\CA)$  and  $\frakF\in\mathbf{Ex}(\frakE)$. We have to prove that
$\mathfrak{PB}_\frakE(\CI)\subseteq \frakF$ if and only if $\CI\subseteq\PEF$.

The inclusion
$$\CI\subseteq \PEF=\{\phi\mid h\phi\in\Ph(\frakF)\hbox{ for all }h\in\Ph(\frakE)\}$$
is equivalent to $\Ph(\frakE)\CI\subseteq \Ph(\frakF)$. Since $\Ph(\frakE)\CI=\PB_\frakE(\CI)$, the last inclusion is equivalent
to $\mathfrak{PB}_\frakE(\CI)\subseteq \frakF$.

The proof for the second pair is similar.
\end{proof}

Using the standard properties of Galois connections we have the following corollary.
\begin{corollary}\label{connection-PB-phantoms}
Let $\CI\in \mathbf{Ideals}(\CA)$  and  $\frakF\in\mathbf{Ex}(\frakE)$.
\begin{enumerate}[{\rm (1)}] \item $\mathfrak{PB}_\frakE(\PEF)\subseteq \frakF$ and
$\CI\subseteq{\Phi}_\frakE(\mathfrak{PB}_\frakE(\CI))$;

\item $\mathfrak{PO}_\frakE(\CPEF)\subseteq \frakF$ and
$\CI\subseteq{\Psi}_\frakE(\mathfrak{PO}_\frakE(\CI))$.
\end{enumerate}
\end{corollary}

The following results exhibit connections between orthogonal ideals and some injective/projective properties.

\begin{proposition}\label{I-perp-ex-PB}
\begin{enumerate}[{\rm (1)}]
  \item Let $\CI\in \mathbf{Ideals}(\CA)$. Then $\CI^\perp=\mathfrak{PB}_\frakE(\CI)\inj$.
  \item Let $\CJ\in \mathbf{Ideals}(\CA)$. Then $^\perp \CJ=\mathfrak{PO}_\frakE(\CJ)\proj$.
\end{enumerate}
\end{proposition}

\begin{proof}
A homomorphism $j:A\to E$ is in $\CI^\perp$  if and only if $(\Sigma j)\Ph(\frakE)\CI=0$.
But $\Ph(\frakE)\CI=\PB_\frakE(\CI)$, and we apply
Lemma \ref{basic-CI-inj} to obtain the conclusion.
\end{proof}

\begin{corollary}\label{PB-inclusions}
\begin{enumerate}[{\rm (1)}]
  \item Let $\CI\in \mathbf{Ideals}(\CA)$.  Then $$\CI\subseteq \Phi_\frakE (\mathfrak{PB}_\frakE(\CI))\subseteq {^{\perp}(\CI^{\perp})},$$ both inclusions becoming equalities when $\CI$ is special precovering.
  \item Let $\CJ\in \mathbf{Ideals}(\CA)$. Then $$\CJ\subseteq \Phi^\frakE (\mathfrak{PO}_\frakE(\CJ))\subseteq ({^\perp\CJ})^{\perp},$$ both inclusions becoming equalities when $\CJ$ is special preenveloping.
\end{enumerate}
\end{corollary}

\begin{proof} We only prove (1).
The first inclusion is
a consequence of Corollary \ref{connection-PB-phantoms}.

For the second inclusion, we replace in Theorem \ref{orth-F} the almost $n$-exact structure $\frakF$ by $\mathfrak{PB}_\frakE(\CI)$,
hence we have  $$\Phi_\frakE (\mathfrak{PB}_\frakE(\CI))\subseteq  {^\perp(\mathfrak{PB}_\frakE(\CI)\textrm{-inj})}.$$
By Proposition \ref{I-perp-ex-PB} we have $\mathfrak{PB}_\frakE(\CI)\textrm{-inj}=\CI^{\perp}$, hence
$\Phi_\frakE (\mathfrak{PB}_\frakE(\CI))\subseteq {^{\perp}(\CI^{\perp})}.$
Finally if $\CI$ is special preevveloping the equality $\CI={^{\perp}(\CI^{\perp})}$ follows by Theorem \ref{cotorsion-precovering}.
\end{proof}

\begin{corollary} Let $\CI,\CJ\in \mathbf{Ideals}(\CA)$.
\begin{enumerate}[{\rm (1)}]
\item If the phantom $\CA$-ideal $\PB_\frakE(\CI)$ is precovering, then
$\Phi_\frakE (\mathfrak{PB}_\frakE(\CI))= {^{\perp}(\CI^{\perp})}$.
\item If the phantom $\CA$-ideal $\PO_\frakE(\CJ)$ is preenveloping, then
$\Phi^\frakE (\mathfrak{PO}_\frakE(\CJ))= {({^{\perp}\CJ})^{\perp}}$.
\end{enumerate}
\end{corollary}

\begin{proof} We only prove (1).
By Theorem \ref{Th-procov-vs-Iinj} we obtain that
there are enough $\mathfrak{PB}_\frakE(\CI)$-injective morphisms. Using Proposition \ref{I-perp-ex-PB} and
Theorem \ref{orth-F} we have
$${^{\perp}(\CI^{\perp})}={^\perp (\mathfrak{PB}_\frakE(\CI)\inj)}={\Phi}_\frakE(\mathfrak{PB}_\frakE(\CI)).$$
\end{proof}

\begin{corollary}\label{phi(F)-precovering}
Let $\frakF$ be an almost $n$-exact structure for $\CA$ such that $\frakF\subseteq \frakE$.
\begin{enumerate}[{\rm (1)}]
  \item Suppose that there are enough $\frakE$-projective morphisms and there are enough
$\frakF$-injective morphisms. Then $\Phi_\frakE(\frakF)$ is a special precovering ideal.
  \item Suppose that there are enough $\frakE$-injective morphisms and there are enough
$\frakF$-projective morphisms. Then $\Phi^\frakE(\frakF)$ is a special preenveloping ideal.
\end{enumerate}
\end{corollary}

\begin{proof} We only prove (1).
By Theorem \ref{orth-F} we know that $\Phi_\frakE(\frakF)={^\perp \frakF\inj}$. But $\frakF\inj$ is a preenveloping ideal,
hence we can apply Theorem \ref{salce-lemma} to obtain the conclusion.
\end{proof}

\subsection{Complete ideal cotorsion pairs}

In order to characterize the ideal cotorsion pairs which are complete, we will study first the existence of some special injective (respectively, projective) preenvelopes (respectively, precovers).

An $\frakH$-injective morphism $e$ is \textsl{special \wrt $\frakE$}
if it can be embedded in a morphism of $n$-angles
\[\xymatrix@=0.5cm{           
  \mathfrak{d}i:& A\ar[r]^e\ar@{=}[d] & E\ar[r]\ar[d]&B_3\ar[r]\ar[d]& \cdots \ar[r]& B_{n-1}\ar[r] \ar[d]   & X\ar[r]  \ar[d]^{\phi}  & \Sigma A \ar@{=}[d]\\
  \mathfrak{d}: & A\ar[r] &A_2\ar[r]& A_3\ar[r]&\cdots \ar[r]& A_{n-1}\ar[r]    & A_n\ar[r]     & \Sigma A         }\]
such that $\mathfrak{d}\in\frakE$ and $\phi\in\Phi_\frakE(\frakH)$.
The notion of \textsl{special $\frakH$-projective homomorphism} is defined dually.

\begin{example}\label{exemple-special-inj}
From Corollary \ref{PB-inclusions} and Proposition \ref{I-perp-ex-PB} we observe that
\begin{enumerate}[{\rm (1)}]
  \item { if $\CI$ is special precovering}, then every special $\CI^\perp$-preenvelope is a special $\mathfrak{PB}_\frakE(\CI)$-injective morphism.
  \item { if $\CJ$ is special preenveloping}, then every special ${^\perp\CJ}$-precover is a special $\mathfrak{PO}_\frakE(\CJ)$-projective morphism.
\end{enumerate}
\end{example}

\begin{proposition}\label{special-injective}
Let $\frakH\in\mathbf{Ex}(\frakE)$. Then
\begin{enumerate}[{\rm (1)}]
\item Every special $\frakH$-injective morphism { is a special $\frakH\text{-{\rm inj}}$-preenvelope} and
a special $\Phi_\frakE(\frakH)^{\perp}$-preenvelope.

\item Every special $\frakH$-projective morphism is a special $\frakH\text{-{\rm proj}}$-precover
and a special $^\perp\Phi^\frakE(\frakH)$-precover.
\end{enumerate}
\end{proposition}

\begin{proof}
{Using Theorem \ref{orth-F} we observe that $\Phi_\frakE(\frakH)\subseteq{^\perp}\frakH\textrm{-inj}$,
 hence every special $\frakH$-injective morphism is a special $\frakH\textrm{-inj}$-preenvelope.}

Let $e$ be a special $\frakH$-injective morphism.
By Corollary \ref{connection-PB-phantoms} we have the inclusion $\mathfrak{PB}_\frakE(\Phi_\frakE(\frakH))\subseteq \frakH$,
hence we can apply
Proposition \ref{I-perp-ex-PB} to obtain
$$e\in \frakH\textrm{-inj}\subseteq \mathfrak{PB}_\frakE(\Phi_\frakE(\frakH))\textrm{-inj}=\Phi_\frakE(\frakH)^{\perp}.$$
Since $\Phi_\frakE(\frakH)\subseteq {^\perp (\Phi_\frakE(\frakH)^\perp)}$ we can apply the definition to obtain that $e$ is a
special $\Phi_\frakE(\frakH)^\perp$-preenvelope.
\end{proof}

The following result provides a method to construct (complete) ideal cotorsion pairs, which  improves Theorem \ref{orth-F}.

\begin{theorem}\label{Iperp=injective} Let $\frakH\in\mathbf{Ex}(\frakE)$.
\begin{enumerate}[{\rm (1)}]
\item If there are enough special $\frakH$-injective morphisms, then
$$(\Phi_\frakE(\frakH),\frakH\textrm{-}\mathrm{inj})$$ is an ideal  cotorsion pair
{which is complete if there are enough $\frakE$-projective morphisms.}

\item If there are enough special $\frakH$-projective morphisms, then
$$(\frakH\textrm{-}\mathrm{proj}, \Phi^\frakE(\frakH))$$ is an ideal cotorsion pair
{which is complete if there are  enough $\frakE$-injective morphisms.}
\end{enumerate}
\end{theorem}

\begin{proof}
Since there are enough special $\frakH$-injective morphisms, it follows that the ideal $\frakH\text{-inj}$ is
a special preenveloping ideal and there are enough $\frakH$-injective morphisms.
Then we can use Theorem \ref{orth-F} to obtain ${^{\perp} \frakH\textrm{-}\mathrm{inj}}=
\Phi_\frakE(\frakH)$.
Now the conclusions are
consequences of Theorem \ref{cotorsion-precovering}.%
\end{proof}

Now we give our main result as follows, which extends \cite[Theorem 5.3.4]{BM} from triangulated categories to $n$-angulated categories.

\begin{theorem}\label{mainthA}
Let $\frakE$ be an almost $n$-exact structure for $\CA$ such that there are enough $\frakE$-injective morphisms
and $\frakE$-projective morphisms, and let $(\CI,\CJ)$ be an ideal cotorsion pair in $\CA$.
Then the following statements are equivalent.
\begin{enumerate}[{\rm (a)}]
 \item $\CI$ is precovering.
 \item $\CI$ is special precovering.
 \item $\CJ$ is preenveloping.
 \item $\CJ$ is special preenveloping.
\item There exists  $\frakH\in\mathbf{Ex}(\frakE)$ with enough special $\frakH$-injective morphisms such that $\CI=\mathrm{\Phi}_\frakE(\frakH)$.
\item There exists  $\frakF\in\mathbf{Ex}(\frakE)$ with enough $\frakF$-injective morphisms such that $\CI=\PEF$.
\item There exists  $\frakH\in\mathbf{Ex}(\frakE)$ with enough special $\frakH$-injective morphisms such that
$\CJ=\frakH\inj$.
\item There exists  $\mathfrak{G}_n\in\mathbf{Ex}(\frakE)$ with enough special $\mathfrak{G}_n$-projective morphisms such that $\CJ=\CPEG$.
\item There exists  $\mathfrak{F}_n\in\mathbf{Ex}(\frakE)$ with enough $\mathfrak{F}_n$-projective morphisms such that $\CJ=\CPEF$.

\item There exists  $\mathfrak{G}_n\in\mathbf{Ex}(\frakE)$ with enough special $\mathfrak{G}_n$-projective morphisms such that
$\CI=\mathfrak{G}_n\proj$.
\end{enumerate}
\end{theorem}

\begin{proof}
We summarize the proof as the following diagram.
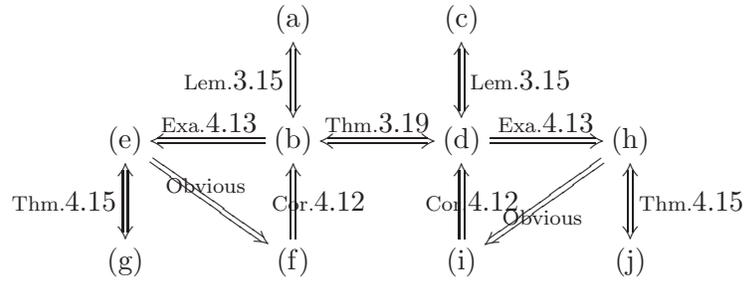
\[
\xymatrix@C=1.5cm@R=1cm{&({\rm a})\ar@{<=>}[d]_{{\rm Lem. } \ref{3.3.2}}&({\rm c})\ar@{<=>}[d]^{{\rm Lem. } \ref{3.3.2}}&\\
({\rm e})\ar@{<=>}[d]_{{\rm Thm. } \ref{Iperp=injective}}\ar@{=>}[rd]^{\!\!\!\!\!\!\!\!\!\!\!\!\rm Obvious}&({\rm b})\ar@{=>}[l]_{{\rm Exa. }\ref{exemple-special-inj}}\ar@{<=>}[r]^{{\rm Thm. }\ref{cotorsion-precovering}}&({\rm d})\ar@{=>}[r]^{{\rm Exa. }\ref{exemple-special-inj}}&({\rm h})\ar@{=>}[ld]^{\!\!\!\!\!\!\!\!\!\!\!\!\rm Obvious}\ar@{<=>}[d]^{{\rm Thm. } \ref{Iperp=injective}}\\
({\rm g})&({\rm f})\ar@{=>}[u]_{\!\!\!\!\!\!\!\!{\rm Cor. }\ref{phi(F)-precovering}}&({\rm i})\ar@{=>}[u]_{\!\!\!\!\!\!\!\!\!\!\!\! {\rm Cor. }\ref{phi(F)-precovering}}&({\rm j})}
\]
\end{proof}


\vspace{0.5cm}

{\bf Acknowledgements.} 
This work was  supported by NSFC (Nos. 11871301,
11901341, 11971225), the project ZR2019QA015 supported by Shandong Provincial Natural Science Foundation,  the project funded by China Postdoctoral Science
Foundation (2020M682141), and the Young Talents Invitation Program of Shandong Province.

 \end{document}